\documentclass[11pt,leqno]{amsproc}

\usepackage{amsfonts}
\usepackage{amsthm}
\usepackage{amsmath}
\usepackage{amssymb}
\usepackage{hyperref}
\theoremstyle{plain}
\newtheorem{theorem}{Theorem}[section]

\newtheorem{assumption}[theorem]{Assumption}
\newtheorem{lemma}[theorem]{Lemma}
\newtheorem{corollary}[theorem]{Corollary}
\newtheorem{proposition}[theorem]{Proposition}
\theoremstyle{remark}
\newtheorem{remark}[theorem]{Remark}
\newtheorem{notation}[theorem]{Notation}
\newtheorem{example}[theorem]{Example}

\def\defn{\mathrel{:=}}
\def\la{\left\lvert}
\def\lA{\left\lVert}
\def\ra{\right\rvert}
\def\rA{\right\rVert}

\def\C{{\mathbb C}}
\def\R{{\mathbb R}}
\def\N{{\mathbb N}}

\def\T{{\mathbb T}}

\def\virgp{\raise 2pt\hbox{,}}

\def\({\left(}
\def\){\right)}
\def\<{\left\langle}
\def\>{\right\rangle}
\def\le{\leqslant}
\def\ge{\geqslant}
\def\les{\lesssim}

\def\Eq#1#2{\mathop{\sim}\limits_{#1\rightarrow#2}}
\def\Tend#1#2{\mathop{\longrightarrow}\limits_{#1\rightarrow#2}}

\def\d{{\partial}}

\def\eps{\varepsilon}

\def\si{{\sigma}}

\def\O{\mathcal O}

\DeclareMathOperator{\RE}{Re}
\DeclareMathOperator{\IM}{Im}

\DeclareMathOperator{\DIV}{\mathrm{div}}
\DeclareMathOperator{\id}{\mathrm{Id}}

\DeclareMathOperator{\cn}{\mathrm{div}}
\DeclareMathOperator{\curl}{\mathrm{curl}}

\def\defn{\mathrel{:=}}

\def\norm#1{\left\lVert{\smash[t]{{#1}}}\right\rVert}

\def\scal#1#2{\langle \, {#1} \hspace{1pt},\hspace{1pt} {#2} \,\rangle}

\numberwithin{equation}{section}

\def\ta{a}

\def\tm{{a}^\eps}

\def\tJ{J^\eps}
\def\tphi{\phi}
\def\tPhi{\Phi}
\def\tu{u^\eps}

\def\tv{v}

\def\tw{\psi^\eps}

\def\tbeta{\beta^\eps}

\def\tg{g^\eps}

\def\tomega{\omega^\eps}
\def\itomega{\overset{\circ}{\omega}{}^\eps}
\def\trho{\rho^\eps}
\def\tp{\rho}
\def\tq{q^\eps}
\def\reste{r^\eps}

\begin{document}

\title[Geometric optics for NLS]{Supercritical geometric optics for
 nonlinear Schr\"odinger equations}   
 \author[T. Alazard]{Thomas Alazard}
\address{CNRS \& Universit\'e Paris-Sud\\Math\'ematiques 
\\ B\^at.~425\\91405 
  Orsay cedex\\France}
\email{Thomas.Alazard@math.cnrs.fr}
\author[R. Carles]{R{\'e}mi Carles}
\address{CNRS \& Universit\'e Montpellier~2\\Math\'ematiques
\\CC~051\\Place Eug\`ene Bataillon\\34095
  Montpellier cedex 5\\ France}
\email{Remi.Carles@math.cnrs.fr}
\thanks{Support by the ANR
  project SCASEN is acknowledged.} 
\begin{abstract}
We consider the small time semi-classical limit for nonlinear Schr\"odinger
equations with defocusing, 
smooth, nonlinearity. For a super-cubic nonlinearity, the
limiting system is not directly hyperbolic, due to the presence of
vacuum. To overcome this issue, we introduce new unknown functions,
which are defined nonlinearly in terms of the wave function
itself. This approach provides a local version of the modulated energy
functional introduced by Y.~Brenier. The system we obtain is
hyperbolic symmetric, and the justification of WKB analysis follows. 
\end{abstract}

\maketitle


\section{Introduction}
\label{sec:intro}

\subsection{Presentation}
\label{sec:pres}

We study the 
behavior of the solution $\tu$ to
\begin{equation}\label{eq:nlssemi}
i\eps \partial_{t} \tu + \frac{\eps^2}{2}\Delta \tu = \la
\tu\ra^{2\sigma}\tu\quad ;\quad \tu_{\arrowvert
  t=0}=a_0^\eps e^{i\phi_{0}/\eps}, 
\end{equation}
as the parameter $\eps\in ]0,1]$ goes to zero. 
To fix matters, we work on $\R^n$, yet all the 
results are valid in the Torus~$\T^n$. Throughout all this
paper, we assume that the space dimension is $n\le 3$, which
corresponds to the physical cases. 
The unknown $\tu$ and the initial amplitude 
$a_0^\eps$ are complex valued, the phase $\phi_0$ is real-valued. The
case of a more general nonlinearity, of the form
\begin{equation*}
  i\eps \partial_t u^\eps +\frac{\eps^2}{2}\Delta
  u^\eps = \eps^\kappa 
  f\left(|u^\eps|^2\right) 
  u^\varepsilon\quad ;\quad 
u^\eps_{\mid t=0} =
  a_0^\varepsilon(x)e^{i\phi_0(x)/\varepsilon},
\end{equation*}
was discussed in \cite{CaBKW}. In particular, WKB type analysis is
justified for $\kappa\ge~1$ (weak nonlinearity). On the other hand,
when $\kappa=0$, there are only two cases in which the mathematical
analysis of the semi-classical limit for nonlinear Schr\"odinger
equations is well developed. First, for analytic initial
data. We refer to \cite{PGX93} and \cite{ThomannAnalytic} for this
approach. Second, for the cubic defocusing nonlinear Schr\"odinger equation
($\si =1$ in \eqref{eq:nlssemi}). 
Our goal is to justify geometric optics in Sobolev spaces for
\eqref{eq:nlssemi} when $\si \ge 2$. This question has remained open
since the pioneering work of E.~Grenier \cite{Grenier98}, where the
nonlinearity had to be cubic at the origin. 
\smallbreak

There are several motivations to study the semi-classical limit for
\eqref{eq:nlssemi}. Let us mention three. First, \eqref{eq:nlssemi}
with $\si=2$ (quintic nonlinearity) is sometimes used as a model for
one-dimensional Bose--Einstein condensation (\cite{KNSQ}). An
external potential is usually considered in this framework (most
commonly, an harmonic potential); we refer to \cite{CaBKW} to show
that the results of the present paper can easily be adapted to that
case (see also \S\ref{sec:potential}). 
\smallbreak

Second, the limit $\eps \to 0$ 
relates classical and quantum wave equations. In particular, the
semi-classical limit $\eps\to 0$ for $u^\eps$ is expected to be 
described by the laws of hydrodynamics; see
e.g. \cite{PGX93,Grenier98,GasserLinMarkowich,DesjardinsLin}.  If we
assume that 
$a_0^\eps \to a_0$ as $\eps \to 0$, then formally, $u^\eps$ is
expected to be well approximated by $a e^{i\phi/\eps}$, where
\begin{equation}
  \label{eq:limite}
  \left\{
\begin{aligned}
&\d_t \phi +\frac{1}{2}|\nabla \phi|^2+ |a|^{2\si}=0
\quad&& ;\quad \phi_{\mid t =0}=\phi_{0},\\
&\d_t a +\nabla\phi\cdot \nabla a +\frac{1}{2}a
\Delta \phi = 0\quad  &&;\quad a_{\mid
  t=0}=a_0.
 \end{aligned}
\right.
\end{equation}
This system is to be understood as a compressible Euler
equation. Indeed, setting $(\rho,v)=(|a|^2,\nabla \phi)$, we see that
$(\rho,v)$ solves:
\begin{equation*}
  \left\{
    \begin{aligned}
     &\d_t v +v\cdot \nabla v + \nabla \(\rho^\si\)=0\quad &&;\quad
v_{\mid t =0}=\nabla\phi_{0},\\ 
& \d_t \rho + \DIV \(\rho v\) =0\quad  && ;\quad \rho_{\mid
  t=0}=|a_0|^2. 
    \end{aligned}
\right.
\end{equation*}
Note that for $\si>1$, the above system is not directly hyperbolic
symmetric, due to the presence of vacuum. We will see that the above
system suffices to describe the convergences of the main two quadratic
observables for $u^\eps$, that is, position and current densities.  
We will also see that passing to the limit $\eps\to 0$ in the usual
conservation laws for nonlinear Schr\"odinger equations, we recover
important conservation laws for the Euler equation 
(see \S\ref{sec:conserv}). This also serves as a background to
note that some blow-up results for the nonlinear Schr\"odinger
equation on the one hand, and the compressible Euler equation on the
other hand, follow from very similar identities (see
\S\ref{sec:blowup}). This remark reinforces the bridge noticed by
D.~Serre \cite{Serre97}.   
\smallbreak

Another motivation lies in the study the Cauchy problem for 
nonlinear Schr\"odinger equations with no small parameter ($\eps =1$ in
\eqref{eq:nlssemi}, typically). As noticed in \cite[Appendix]{BGTENS}
and \cite{CaARMA}, one can prove ill-posedness results 
for energy super-critical equations by 
reducing the problem to semi-classical analysis
for \eqref{eq:nlssemi}. In \cite[Appendix~C]{CaARMA}, a result of loss of
regularity was proved for the cubic, defocusing nonlinear
Schr\"odinger equation, in the spirit of the pioneering work of
G.~Lebeau \cite{Lebeau05}. It concerned the flow associated to the
nonlinear Schr\"odinger equation near the origin. This was extended in
\cite{CaInstab} to the case of data of arbitrary size in Sobolev spaces. When
the nonlinearity is defocusing and not necessarily cubic, the result
of \cite[Appendix~C]{CaARMA} was extended in \cite{AC-perte}, by
studying the semi-classical limit for \eqref{eq:nlssemi}. However,
\cite{AC-perte} does not use the complete justification of geometric
optics, which makes it impossible to extend the results in
\cite{CaInstab} to the case of super-cubic nonlinearities. The
analysis presented in this paper makes it possible.

\subsection{Main results}
\label{sec:main}

For $s\ge 0$, we shall denote $H^s(\R^n)$, or simply $H^s$, the
Sobolev space of order $s$. We equip
$H^\infty(\R^n)=H^\infty \defn \cap_{s>0}H^s(\R^n)$ with the distance
\begin{equation*}
  d(f,g)= \sum_{s\in \N} 2^{-s}\frac{\|f-g\|_{H^s}}{1+\|f-g\|_{H^s}}\cdot
\end{equation*}
Note that for any $k\in \N$ and any interval $I$, $C^k(I;H^\infty) =
\cap_{s\ge 0}C^k(I; H^s)$. 
\begin{assumption}\label{assu:general}
We require $\sigma\in\N\setminus \{0\}$ 
without recalling this assumption explicitly
in the statements. Similarly, it is assumed that $a_0, \phi_0 \in
H^\infty$, where recall that $\phi_{0}$ is 
real-valued.
We also suppose that 
$a_0^\eps$ belongs uniformly to $H^\infty$ and 
converges towards $a_0$ in $H^\infty$ as $\eps \to 0$. More
precisely,
\begin{equation*}
  a_0^\eps =a_0 +\O(\eps)\quad \text{in }H^s(\R^n),\ \forall s\ge 0. 
\end{equation*}
\end{assumption}
The first remark, based on a change of unknown due to T.~Makino, S.~Ukai and
S.~Kawashima~\cite{MUK86} (see also \cite{JYC90}), is that the limiting system
\eqref{eq:limite} is locally well-posed in Sobolev spaces, despite the
possible presence of vacuum:
\begin{lemma}[from \cite{AC-perte}]\label{weaklemma:MUK}
Let $n\ge 1$, and let Assumption~\ref{assu:general} be
satisfied. There exists 
$T^*>0$ such that \eqref{eq:limite} has a
unique maximal solution 
$(\phi,a)$ in $C([0,T^*[;H^\infty(\R^n))$.  
\end{lemma}
The proof is recalled in $\S\ref{sec:prelim}$. 
It is based on a change of unknown introduced in \cite{MUK86} (see
also \cite{JYC90}), which
makes it possible  
to rewrite the equation under the form of a quasi-linear symmetric
hyperbolic system.  
This transformation of the equations, which consists in introducing
$(v,u)\defn (\nabla\phi,a^\sigma)$,  
clearly exhibits a key dichotomy between $\sigma=1$ and 
$\sigma\ge 2$. In particular, a stability analysis in the case $\si
\ge 2$ is not straightforward, since the above mentioned change of
variables does not seem to be well adapted to Schr\"odinger
equations. 

\smallbreak
Here is the main result of this paper. In the context of Assumption
\ref{assu:general}, we prove 
that the solutions of~\eqref{eq:nlssemi} exist and satisfy uniform
estimates on a time interval 
which is independent of $\eps$. 
\begin{theorem}\label{theo:main}
Let $n\le 3$, and let Assumption~\ref{assu:general} be satisfied. 
There exists $T\in ]0,T^*[$, where $T^*$ is given by
Lemma~$\ref{weaklemma:MUK}$, such that the following holds. 
For all $\eps\in]0,1]$ the Cauchy problem \eqref{eq:nlssemi} has a unique 
solution $\tu\in C([0,T];H^\infty(\R^n))$. Moreover,
\begin{equation}\label{eq:uniform}
\sup_{\eps \in ]0,1]}\bigl\lVert u^\eps
e^{-i\phi/\eps}\bigr\rVert_{L^\infty([0,T];H^k(\R^n))}<+\infty,
\end{equation}
where $\phi\in C^\infty_{b}([0,T]\times\R^n)$ is given by
\eqref{eq:limite}, 
and the index $k$ is as follows:
\begin{itemize}
\item If $\si =1$, then $k\in \N$ is arbitrary.
\item If $\si = 2$ and $n=1$, then we can take $k=2$.
\item If $\si = 2$ and $2\le n\le 3$, then we can take $k=1$.
\item If $\si\ge 3$, then we can take $k=\si$. 
\end{itemize}
\end{theorem}
\begin{remark}
  The estimate \eqref{eq:uniform} is trivial for $k=0$, from the
  conservation of mass, 
  which holds even for weak solutions (\cite{GV85c}). 
\end{remark}
\begin{remark} For sufficiently large $\si$, the approach followed in
  this paper makes it possible to extend Theorem~\ref{theo:main} to
  the case of higher dimensions, $n\ge 4$. We shall not pursue this
  question. 
\end{remark}
\begin{remark}
The assumption $a_0^\eps =a_0 +\O(\eps)$ plays a crucial role in the
above result. Indeed, the analysis 
in \cite{CaARMA} shows that in the case $\si=1$, if we assume only
$a_0^\eps =a_0 + o(1)$, then the conclusion of Theorem~\ref{theo:main}
fails. For instance, if $a_0^\eps = (1+\eps^\alpha)a_0$ for some
$0<\alpha \le 1$, then for arbitrarily small $t>0$ independent of
$\eps$, $u^\eps e^{-i\phi/\eps}$ has oscillations of order
$\eps^{1-\alpha}$. So if $\alpha <1$, then $u^\eps e^{-i\phi/\eps}$ is
not bounded in $H^1$. 
\end{remark}
For $\si=1$, the above result is a consequence of the analysis due to
E.~Grenier \cite{Grenier98}, and remains valid in any space dimension
$n\ge 1$. We propose an alternate proof in \S\ref{sec:cubic}.  

In the quintic case $\si =2$,  for all $\eps>0$, the Cauchy problem
\eqref{eq:nlssemi} has a unique global solution  
in  $C(\R;H^\infty(\R^n))$. Indeed, for $n=1$ this follows from
standard results for semi-linear equations;  
in the energy-subcritical case $n= 2$, this follows from 
Strichartz estimate and the conservation laws; 
for the difficult energy-critical case $n=3$, this has been proved 
by J. Colliander, M. Keel, G. Staffilani, H. Takaoka and T. Tao
in~\cite{CKSTTAnnals}. Therefore, the main point in our result is the
uniform bound \eqref{eq:uniform}. The same is true for the case $n\le
2$ and $\si\ge 3$, since the nonlinearity is then $H^1$ sub-critical.

For $\sigma\ge 3$ and $n=3$, the equation is $H^1$
super-critical. Therefore, not only the bound  \eqref{eq:uniform} is
new, but also the fact that we can construct a smooth solution $\tu$ to
\eqref{eq:nlssemi} on a time interval $[0,T]$ independent of $\eps\in
]0,1]$. 
\smallbreak

As a matter of fact, the proof that we present for
Theorem~\ref{theo:main} yields a stronger result, whose statement
demands a more precise analysis (see Theorem~\ref{theo:tout}). In
particular, we infer that the quadratic observables converge strongly towards
the solution of compressible Euler equations  
for potential flows in vacuum, hence giving the Wigner measure
associated to $(u^\eps)_\eps$ (see e.g. \cite{BurqMesures,GMMP} for
the definition and the main properties). The following result is
proved in \S\ref{sec:quad}.
\begin{proposition}\label{prop:quad}
  Let $n\le 3$, and let Assumption~\ref{assu:general} be satisfied. 
There exists $T\in ]0,T^*[$, where $T^*$ is given by
Lemma~$\ref{weaklemma:MUK}$, such that the position and current
densities converge strongly on $[0,T]$ as $\eps \to 0$:
\begin{align*}
  &|u^\eps|^2 \Tend \eps 0 |a|^2\quad && \text{in
  }C\([0,T];L^{\si+1}(\R^n)\). \\
& \IM \(\eps \overline u^\eps \nabla u^\eps\)\Tend \eps 0
  |a|^2\nabla \phi \quad
  && \text{in 
  }C\([0,T];L^{\si+1}(\R^n)+ L^{1}(\R^n)\). 
\end{align*}
In particular, there is only one Wigner measure associated to
$(u^\eps)_\eps$, and it is given by
\begin{equation*}
  \mu(t,dx,d\xi) = |a(t,x)|^2 dx\otimes \delta \( \xi - \nabla
  \phi(t,x)\). 
\end{equation*}
\end{proposition}

The analysis proposed to prove Theorem~\ref{theo:main} allows us to
compute the leading order 
behavior of the wave function $\tu$, provided that we know a 
more precise WKB expansion of the initial amplitude.
\begin{assumption}\label{assu:ghost}
In addition to Assumption~\ref{assu:general}, we
assume that there exists $a_1\in H^\infty(\R^n)$ such that
\begin{equation*}
  a_0^\eps = a_0 +\eps a_1 +\O\(\eps^2\)\quad \text{in }H^s(\R^n),\
  \forall s\ge 0. 
\end{equation*}
\end{assumption}
\begin{theorem}\label{theo:ghost}
  Let $n\le 3$, and let 
  Assumption~\ref{assu:ghost} be satisfied.  There exists
  $\widetilde a\in  C([0,T^*[;H^\infty)$, and for any $T\in
  ]0,T^*[$, there exists $\eps(T)>0$, such that $u^\eps \in
  C([0,T];H^\infty)$ for $\eps \in ]0,\eps(T)]$, and
\begin{equation}\label{eq:pointwise}
  \begin{aligned}
    &\left\| u^\eps - \widetilde a
  e^{i\phi/\eps}\right\|_{L^\infty([0,T];L^2\cap L^p)}=\O(\eps)\quad
  \text{when 
  }\si=2 \text{ and }2\le n\le 3,\\ 
&\left\| u^\eps - \widetilde a
  e^{i\phi/\eps}\right\|_{L^\infty([0,T];L^2\cap
  L^\infty)}=\O(\eps)\quad \text{in the other cases,}
  \end{aligned}
\end{equation}
where $p$ is such that $H^1(\R^n)\subset L^p(\R^n)$. 
\end{theorem}
\begin{remark}
In general, $\widetilde a\not = a$, unless $a_0$ is real-valued and 
$a_1 \in i\R$ (see \S\ref{sec:ghost}). Therefore, the system
\eqref{eq:limite} does not suffice, in general, to describe the
asymptotic behavior of the \emph{wave function} $u^\eps$, even though
it suffices to describe the position and current densities (see
Proposition~\ref{prop:quad} above). 
\end{remark}


\subsection{Scheme of the proof of Theorem~\ref{theo:main}}
\label{sec:scheme}

To prove that the solutions to the
Cauchy problem~\eqref{eq:nlssemi} exist for a time  
independent of $\eps$, it is enough to prove uniform estimates for the
$L^\infty$ norm of $\tu$ (see Lemma~\ref{lemI1} below).  
To do so, our approach toward the semi-classical limit is to filter
out the oscillations  
by the following change of unknown, involving the solution $(a,\phi)$ of the
limit system~\eqref{eq:limite}: 
\begin{equation}\label{defa}
\tm(t,x)\defn u^\eps(t,x) e^{-i\phi(t,x)/\eps}.
\end{equation}
The key point is that, although it is obviously equivalent to prove
$L^\infty$ estimates for $\tu$ and $\tm$, it is  
expected that one can prove uniform estimates in Sobolev spaces for
$\tm$, thereby obtaining the desired  
$L^\infty$ estimates from the Sobolev embedding. Obviously, uniform
estimates in Sobolev spaces for $u^\eps$ are not expected to hold, due
to the rapid oscillations described by $\phi$. 

\smallbreak
The amplitude $\tm$ solves the following evolution equation:
\begin{equation}\label{eq:ae}
\left\{
  \begin{aligned}
&  \partial_t \tm +\nabla\phi\cdot\nabla \tm
+\frac{1}{2}\tm \Delta\phi
- i\frac{\eps}{2}\Delta \tm
=-\frac{i}{\eps}\( \left| \tm \right|^{2\si} - \left|a\right|^{2\si}\)
\tm . \\
&\tm_{\mid t=0}= a_0^\eps.  
  \end{aligned}
\right.
\end{equation}
It is clear that the mass is conserved:
$$
\lA \tm(t)\rA_{L^2}=\lA \tu(t)\rA_{L^2}=\lA \tu(0)\rA_{L^2}=\lA
a_0^\eps\rA_{L^2}.
$$ 
This can be seen 
by multiplying~\eqref{eq:ae} by $\overline a^\eps$, taking the real part
and integrating over $\R^n$. 
Note that the large term in $\eps^{-1}$ disappears from 
the energy estimate. 
Indeed, the large term in $\eps^{-1}$ is
a nonlinear rotation term. But precisely because this term is
nonlinear, it does not disappear from the estimate
of the derivatives  
(the equation is not translation invariant). Indeed, 
$\nabla\tm$ solves
\begin{equation}\label{eq:nablatm}
\begin{aligned}
&\Bigl(\partial_t +\nabla\phi\cdot\nabla
+\frac{1}{2} \Delta\phi - i\frac{\eps}{2}\Delta\Bigr) \nabla\tm+
\nabla\tm\cdot\nabla\nabla\phi+\frac{1}{2}\tm\nabla\Delta\phi 
\\
&\qquad+\frac{i}{\eps} \( \left| \tm \right|^{2\si} - \left|
  a\right|^{2\si}\) \nabla\tm 
+\frac{i}{\eps}\tm \nabla\( \left| \tm \right|^{2\si} - \left|
  a\right|^{2\si}\)=0.  
\end{aligned}
\end{equation}
This equation is of the form
$$
\bigl(\partial_t +
L(v,\partial_{x})+\mathcal{L}(\eps,\partial_{x})\bigr)\nabla\tm 
+ \frac{i}{\eps}\tm \nabla\( \left| \tm \right|^{2\si} - \left|
  a\right|^{2\si}\)=0, 
$$
where $\mathcal{L}(\eps,\partial_{x})$ is skew symmetric. Again, 
by multiplying~\eqref{eq:nablatm} by $\nabla \overline a^\eps$, taking
the real part and integrating over $\R^n$, we obtain 
\begin{align*}
&\frac{1}{2}\frac{d}{dt}\lA \nabla \tm\rA_{L^2}^2 
-\frac{1}{\eps}\int \DIV (\IM(\overline a^\eps\nabla \tm)) \( \left| \tm
\right|^{2\si} - \left| a\right|^{2\si}\)\\ 
&\quad=-\RE \int_{\R^n} \bigl( \nabla\tm\cdot\nabla\nabla\phi
+\frac{1}{2}\tm\nabla\Delta\phi\bigr)\nabla \overline a^\eps\, 
dx. 
\end{align*}
Together with the mass conservation, this yields the following
identity for the energy $E^\eps\defn \lA\tm\rA_{H^1}^2$:
\begin{equation*}
\frac{1}{2}\frac{d E^\eps}{dt} 
-\frac{1}{\eps}\int \DIV (\IM(\overline a^\eps\nabla \tm)) \( \left| \tm
\right|^{2\si} - \left| a\right|^{2\si}\) 
\le C_\phi E^\eps, 
\end{equation*}
for some constant $C_\phi$
depending only on the known solution $(a,\phi)$ of the limit system. 
The idea is then to find a second energy functional $\mathcal{E}^\eps$
such that 
\begin{equation}\label{secEnergy}
\frac{1}{2}\frac{d\mathcal{E^\eps}}{dt}
+\frac{1}{\eps}\int \DIV (\IM(\overline a^\eps\nabla \tm)) \( \left| \tm
\right|^{2\si} - \left| a\right|^{2\si}\) 
\le C_{a,\phi}(E^\eps+\mathcal{E}^\eps).
\end{equation}
By adding the two inequalities, one obtains a uniform in $\eps$ energy
estimate  
$$
E^\eps(t)+\mathcal{E}^\eps(t) \le  e^{C_{a,\phi}t}
(E^\eps(0)+\mathcal{E}^\eps(0)). 
$$
The previous strategy has many roots. For the semi-classical limit, this 
goes back to the work of Y.~Brenier~\cite{BrenierCPDE},
P.~Zhang~\cite{ZhangSIMA}, F.~Lin and P.~Zhang~\cite{LinZhang}, and is
referred to as  
a modulated energy estimate. Here, we will get the same
result in a different way. Our approach   
amounts to trying to find a nonlinear change of unknown to
symmetrize the equations.  
We will define $\tg$ and $\tq$ such that
\begin{equation*}
\partial_{t}\tq+\tg\DIV( \IM(\overline a^\eps\nabla\tm) ) 
+\nabla\phi\cdot\nabla\tq + \frac{\sigma+1}{2}\tq\Delta\phi =0,
\end{equation*}
and
$$
\tq\tg=\frac{1}{\eps}\(\left| \tm \right|^{2\si} - \left| a\right|^{2\si}\).
$$
Not only does this allow to obtain~\eqref{secEnergy} with
$\mathcal{E}^\eps\defn \lA \tq\rA_{L^2}^2$,  
but also to derive uniform estimates in Sobolev spaces. More
precisely, we will see that the system of equations satisfied by
$(\tm, \nabla \tm, \tq)$ is essentially hyperbolic symmetric
(plus some skew-symmetric terms). Therefore, we can derive energy
estimates, which in turn imply Theorem~\ref{theo:main}. Note that the
idea of introducing new unknown functions to diminish the complexity
of the initial problem is a strategy that has proven successful in many
occasions: for instance, blow-up for the nonlinear wave equation,
\cite{AlinhacAJM95} (see also \cite{Alinhac,AlinhacForges}), 
low Mach number limit of the full Navier-Stokes equations
\cite{ThomasARMA}, or 
geometric optics for the incompressible Euler or Navier-Stokes equations
\cite{CheverryCMP,CheverryBullSMF,CG05}.

\section{Preliminaries}
\label{sec:prelim}

Since, for $\sigma\in\N$, the nonlinearity in~\eqref{eq:nlssemi} is
smooth, the usual theorems for semi-linear evolution 
equations (see e.g. \cite{CazHar}) imply the following result.
\begin{lemma}\label{lemI1}
Let $\sigma,n \in\N\setminus\{0\}$. For (fixed) $\eps\in ]0,1]$, assume that
$u^\eps_{\mid t=0}\in~H^s(\R^n)$ with $s>n/2$. Then there 
exists $T^\eps$ such that \eqref{eq:nlssemi} 
has a unique maximal solution $u^\eps\in C([0,T^\eps[;H^{s}(\R^n))$: 
if $T^\eps<+\infty$, then 
\begin{equation}\label{continuation}
\lim _{t\rightarrow T^\eps}\lA u^\eps(t)\rA_{L^\infty(\R^n)}=+\infty.
\end{equation}
Consequently, if $u^\eps(0)\in H^\infty(\R^n)$, then $u^\eps\in
C^\infty([0,T^\eps[;H^\infty(\R^n))$. 
\end{lemma}
With regards to the limit system \eqref{eq:limite}, we recall 
the proof of Lemma~\ref{weaklemma:MUK}.
\begin{lemma}\label{lemma:MUK}
Let $\sigma\in\N$ and $n\ge 1$. For all $(\phi_{0},a_0)\in
H^{s+1}(\R^n)\times H^{s}(\R^n)$ with $s>n/2+1$,  
there exists $T^*>0$ such that \eqref{eq:limite} 
has a unique maximal solution $(\phi,a)$ in
$C([0,T^*[;H^{s+1}(\R^n)\times H^{s-1}(\R^n))$.  
In addition, if $\phi_{0},a_0\in H^\infty(\R^n)$, then $\phi,a\in
C^\infty([0,T^*[;H^\infty(\R^n))$. 
\end{lemma}
\begin{remark} The lifespan $T^*$ is finite for all compactly support
  initial data (see Proposition~\ref{prop:blowup}). 
If $\sigma=1$, then $a$ belongs to
$C([0,T^*[;H^s(\R^d))$ as soon as 
$(\phi_{0},a_0)\in H^{s+1}(\R^n)\times H^{s}(\R^n)$. 
What makes the previous result non-trivial is the presence of vacuum
when $\si \ge 2$:
at the zeroes of $a$, \eqref{eq:limite} ceases to be hyperbolic, and
this may cause a loss of regularity. 
\end{remark}
\begin{proof}[Sketch of the proof]
One can transform \eqref{eq:limite} into a quasi-linear system by
differentiating the equation for $\phi$: with  
$v=\nabla\phi$, one has
\begin{equation}
  \label{II.2bis}
  \left\{
    \begin{aligned}
     &\d_t v +v\cdot \nabla v+ \nabla |a|^{2\sigma} =0\quad &&; \quad
     v_{\mid t =0}=\nabla\phi_{0},\\ 
     &\d_t a +v\cdot \nabla a +\frac{1}{2}a \DIV v = 0\quad &&;\quad
     a_{\mid t=0}=a_0. 
\end{aligned}
\right.
\end{equation} 
For the cubic case where $\sigma=1$, this system enters 
the standard framework of quasi-linear symmetric hyperbolic systems,
with a constant symmetrizer.  
Thus, one can solve the Cauchy problem~\eqref{eq:limite} in standard
fashion: one first solves  
\eqref{II.2bis} and then checks that $\curl v=0$, so that $v=\nabla
\phi$ for some $\phi$.  
By contrast, for $\sigma>1$, System~\eqref{II.2bis} is no longer
symmetric. However, as in~\cite{MUK86},  
one can prove that the Cauchy problem for \eqref{II.2bis} 
is well-posed, with loss 
of (at most) one derivative for $a$, by introducing
$A=a^\sigma$. Indeed, $(v,A)$ solves a quasi-linear 
hyperbolic system with constant symmetrizer:
\begin{equation}
  \label{II.2ter}
  \left\{
    \begin{aligned}
     &\d_t v +v\cdot \nabla v+ \nabla |A|^{2} =0\quad &&; \quad
     v_{\mid t =0}=\nabla\phi_{0},\\ 
     &\d_t A +v\cdot \nabla A +\frac{\sigma}{2}A
  \DIV v = 0\quad 
     &&;\quad  A_{\mid t=0}=a_0^\sigma. 
\end{aligned}
\right.
\end{equation} 
This allows us to determine~$v$, and hence $\phi$, by setting
\begin{equation*}
  \phi(t,x) = \phi_0(x) - \int_0^t\( \frac{1}{2}|v(\tau,x)|^2 +
  |A(\tau,x)|^2\) d\tau.
\end{equation*}
Then $\d_t \(\nabla \phi -v\)= \nabla \d_t \phi -\d_t v=0$, hence
$v=\nabla \phi$. 
Once this is
granted, one can define $a$ as the solution of  
the second equation in \eqref{II.2bis}, where $v$ is now viewed as a
given coefficient. 
Since $A$ and $a^\si$ satisfy the same linear equation, with identical
initial data, we obtain $A=a^\sigma$. Therefore, $(a,\phi)$ solves
\eqref{II.2bis}. Finally, the local existence 
time $T^*$ may be 
chosen independent of $s>n/2+1$, thanks to tame estimates (see
e.g. \cite{Taylor3}). 
\end{proof}
For further references, we conclude this paragraph by recalling a standard 
estimate in Sobolev spaces for systems of the form
\begin{equation}\label{system:hyp}
\partial_{t}U+ \sum_{1\le j\le n}A_{j}(\Phi,U)\partial_{j}U
+\eps \mathcal{L}(\partial_{x})U=E(\Phi,U), 
\end{equation}
where $U\colon [0,T]\times\R^n\to \C^d$ with $d\ge 1$, $\eps\in\R$ and: 
\begin{itemize}
\item $\Phi\colon [0,T]\times\R^n \to \C^d$ is a given function.
\item The $A_{j}$'s
are $d\times d$ Hermitian matrices depending smoothly on their
arguments. 
\item $\mathcal{L}(\partial_{x})=\sum L_{jk}\partial_j\partial_k$  
is a skew-symmetric second order differential operator with constant
coefficients. 
\item $E$ a $C^\infty$ function of its arguments, vanishing at the origin. 
\end{itemize}
\begin{lemma}\label{lemma:hyp}
Let $n\ge 1$ and $s>n/2+1$. 
There exists a smooth non-decreasing function $C$ from $[0,+\infty[$
to $[0,+\infty[$ such that,  
for all $T>0$, all $\eps\in\R$, all coefficient $\Phi\in
C([0,T];H^s(\R^n))$ and  
all unknown $U\in C([0,T];H^s(\R^n))$ satisfying 
\eqref{system:hyp}, there holds
\begin{equation*}
\sup_{t\in [0,T]}\lA U(t)\rA_{H^s}\le \lA U(0)\rA_{H^s} e^{C(M)T},
\end{equation*}
with $M\defn \lA \Phi\rA_{L^\infty([0,T];H^s(\R^n))}$.
\end{lemma}
\begin{proof}
We want to estimate the $L^2(\R^n)$ norm of $\Lambda^{s}U$, 
where $\Lambda^{s}$ is the Fourier multiplier
$(\id-\Delta)^{s/2}$. To deal with smooth functions, we use the 
Friedrichs mollifiers: let $\jmath\in C^\infty_{0}(\R^n)$ be such that
$\jmath(\xi)=1$ for $|\xi|\le 1$, then we  
define $J_{\delta}=\jmath(\delta D_{x})$ as the Fourier multiplier
with symbol $\jmath(\delta\xi)$.  

With these notations, set $U_{\delta}\defn
J_{\delta}\Lambda^{s}U$. Since $s-1>n/2$,  
$H^{s-1}(\R^n)$ is an algebra which is stable by composition ($F(u)\in
H^{s-1}(\R^n)$ whenever $u\in H^{s-1}(\R^n)$ and  
$F\in C^\infty$ satisfies $F(0)=0$):
$U\in~C^1([0,T];H^{s-2}(\R^n))$. Therefore, $U_{\delta}$ is smooth:
$U_{\delta}\in~C^1([0,T];H^\infty(\R^n))$. Now write 
$$
\partial_{t}U_{\delta}
+\sum_{1\le j\le n}
A_{j}(\Phi,U)\partial_{j}U_{\delta}+\eps\mathcal{L}(\partial_{x})
U_{\delta}=f_\delta, 
$$
with 
$$
f_\delta= \sum_{1\le j\le
  n}[A_{j}(\Phi,U),J_{\delta}\Lambda^s]\partial_{j}U+ 
J_{\delta}\Lambda^s E(\Phi,U). 
$$ 
Since $\mathcal{L}(\partial_{x})=-\mathcal{L}(\partial_{x})^{*}$, and
since  $U_{\delta}\in C^{1}([0,T];L^2(\R^n))$, by taking the inner
product in $L^2(\R^n)$,  we get
\begin{align*}
\frac{d}{dt}
\norm{U_{\delta}}_{L^{2}}^{2}
&=\sum_{1\le j\le
  n}\scal{\partial_{j}A_{j}(\Phi,U)U_{\delta}}{U_{\delta}} 
+2\scal{f_\delta}{U_{\delta}} \\
&\le \Big(1+\sum_{1\le j\le n} \lA
\partial_{j}A_{j}(\Phi,U)\rA_{L^\infty} \Bigr) 
\bigl\Vert U_{\delta}\bigr\rVert_{L^2}^2 +\lA f_\delta\rA_{L^2}^2, 
\end{align*}
where we have used the symmetry of the matrices $A_{j}$. 
The Sobolev embedding and the usual nonlinear estimates (see
\cite{Taylor3}) imply  
\begin{align*}
\lA \partial_{j}A_{j}(\Phi,U)\rA_{L^{\infty}}
&\le C(\lA (\Phi,U)\rA_{W^{1,\infty}})\le 
C(\lA (\Phi,U)\rA_{H^{s}}),\\
\lA [A_{j}(\Phi,U),J_{\delta}\Lambda^{s}]\partial_{j}U\rA_{L^{2}} 
&\le K \norm{\widetilde{A}_{j}(\Phi,U)}_{H^{s}}
\lA \partial_{j}U\rA_{H^{s-1}}\le C(\lA (\Phi,U)\rA_{H^{s}}),\\
\lA  J_{\delta}\Lambda^s E(\Phi,U)\rA_{L^2}&\le K \lA E(\Phi,U) \rA_{H^s}
\le C(\lA (\Phi,U)\rA_{H^{s}}),
\end{align*}
where $\widetilde{A}_{j}=A_{j}-A_{j}(0)$ and $C$ denotes a smooth
non-decreasing function independent of $\delta$.  
To complete the proof, apply Gronwall lemma and let $\delta$ go to
$0$ in the inequality  
thus obtained.
\end{proof}

\section{Proof of Theorem~\ref{theo:main} in the case $\si=1$}
\label{sec:cubic}
Recall that $\tm$ is defined as:
\begin{equation*}
\tm(t,x)\defn u^\eps(t,x) e^{-i\phi(t,x)/\eps},
\end{equation*}
where $\phi\in C^\infty([0,T^*[\times\R^n)$ is given by \eqref{eq:limite}. 
Assume in the rest of this paragraph that
$\si=1$. Then, \eqref{eq:ae} reads
\begin{equation*}
\left\{
  \begin{aligned}
&  \partial_t \tm +\nabla\phi\cdot\nabla \tm
+\frac{1}{2}\tm \Delta\phi
- i\frac{\eps}{2}\Delta \tm
=-\frac{i}{\eps}\( \left| \tm \right|^{2} - \left|a\right|^{2}\)
\tm . \\
&\tm_{\mid t=0}= a_0^\eps.  
  \end{aligned}
\right.
\end{equation*}
Let $s>n/2+1$ and set $\tau^\eps\defn \min (T^*,T^\eps)$, where $T^*$
and $T^\eps$ are   
given by Lemmas~\ref{lemI1} and~\ref{lemma:MUK}. 
We prove that there exists a function~$C$ from $[0,+\infty[$ to $[0,+\infty[$ 
such that, 
for all~$\eps\in ]0,1]$ and all~$t\in [0,\tau^\eps[$,
\begin{equation}\label{Us1}
\lA \tm (t) \rA_{H^s}\le \lA \tm(0)\rA_{H^s}e^{t C(M^\eps(t))},
\end{equation}
where 
$$
M^\eps(t)\defn \lA \tm\rA_{L^\infty([0,t];H^s(\R^n))} 
+ \lA (a,\phi)\rA_{L^\infty([0,t];H^{s+3}(\R^n))}.
$$
This suffices to conclude by a standard continuity argument. 
Indeed, set 
$$
M_0\defn  \sup_{\eps \in
  ]0,1]} \lA a_0^\eps\rA_{H^s} + \lA
(a,\phi)\rA_{L^\infty([0,T^{*}/2];H^{s+3}(\R^n))}<+\infty,  
$$
and choose $T_{0}\in ]0,T^*/2]$ small so that 
$M_{0}\exp(T_{0}C(2M_{0}))<2M_{0}$. Since $M^{\eps}(0)<2M_{0}$ and since 
$M^\eps\in C^0([0,\tau^\eps[)$, \eqref{Us1}  implies 
\begin{equation*}
  M^\eps(t) <2M_{0},\quad \forall t\in [0,\min\{T_{0},T^\eps\}[. 
\end{equation*}
Sobolev embedding then shows that $\lA \tu (t)
\rA_{L^\infty}=\lA \tm (t) \rA_{L^\infty}$ is uniformly bounded for all 
$\eps\in ]0,1]$ and all $t\in [0,\min\{T_{0},T^\eps\}[$.  
Hence, the continuation principle~\eqref{continuation} 
implies that $T^\eps\ge T_0>0$ for all $\eps\in ]0,1]$. The estimate
\eqref{eq:uniform}  
with $\si=1$ then follows from the bound $\displaystyle \sup_{\eps\in]0,1]}
\sup_{t\in[0,T_0]}M^\eps(t)\le 2M_{0}$.

\smallbreak

Theorem~\ref{theo:main} for $\si=1$ was first established by
E.~Grenier in \cite{Grenier98}, whose approach  
is based on a subtle phase/amplitude representation of the solution. 
Here, we give an alternate proof which consists 
in symmetrizing the large terms in $\eps^{-1}$ in the equation for
$\tm$ by introducing  
$$
q^\eps \defn \frac{\la\tm\ra^2-\la a\ra^2}{\eps}\cdot
$$ 
We find directly, in view of Assumption~\ref{assu:general}: 
\begin{equation*}
\partial_{t}\tq + \DIV \(\IM \(\overline a^\eps\nabla\tm\)\) + \cn
(\tq\nabla\phi)=0\ ;\  \|q^\eps_{\mid
  t=0}\|_{H^s(\R^n)}=\O(1),\ \forall s\ge 0.
\end{equation*}
Furthermore, with this notation the equations for $\tm$ and
$\tw\defn\nabla\tm$ read 
\begin{equation*}
\left\{
  \begin{aligned}
   &\partial_t \tm +\nabla\phi\cdot\nabla \tm
+\frac{1}{2}\tm \Delta\phi
- i\frac{\eps}{2}\Delta \tm
+i q^\eps \tm=0,\\
&\partial_t \tw +\nabla\phi\cdot\nabla \tw
+\frac{1}{2}\tw \Delta\phi +
\tw\cdot\nabla\nabla\phi+\frac{1}{2}\tm\nabla\Delta\phi  \\
&\qquad + i q^\eps \tw + i\tm \nabla q^\eps=i\frac{\eps}{2}\Delta
\tw.  
  \end{aligned}
\right.
\end{equation*}
It is easily verified that $U^\eps\defn
(2q^\eps,\tm,\overline a^\eps,\tw,\overline\tw)\in
C^\infty([0,\tau^\eps[;H^\infty(\R^n))$  
satisfies a system of the form~\eqref{system:hyp}, 
that is
$$
\partial_{t}U^\eps+ \sum_{1\le j\le n}A_{j}(\Phi,U^\eps) 
\partial_{j}U^\eps +\eps L(\partial_{x})U^\eps=E(\Phi,U^\eps), 
$$
where $\Phi\defn (\nabla\phi,\Delta\phi,\nabla\Delta\phi)$. 
Hence, by Lemma~\ref{lemma:hyp}, 
we obtain the desired estimate~\eqref{Us1} and conclude the proof 
of Theorem~\ref{theo:main} in the case $\si=1$.

\section{The case $\si\ge 2$}
\label{sec:reste}

We now follow the strategy presented in \S\ref{sec:scheme}. We introduce a
nonlinear change of unknown functions which, together with
\eqref{eq:ae}, yields a quasi-linear system of the form
\eqref{system:hyp}. We conclude the proof of Theorem~\ref{theo:main}
and Proposition~\ref{prop:quad} 
thanks to a general result on the composition by non-smooth functions
in Sobolev spaces.

\subsection{A nonlinear change of
  variable}\label{sec:nonlinear}
As already explained, to symmetrize the equations, 
our idea is to split the term $\left| a^\eps \right|^{2\si} -
\left|\ta\right|^{2\si}$ as a product  
\begin{equation*}
 \left| a^\eps \right|^{2} -
\left|\ta\right|^{2\si}= \tg\tbeta=(G\, B)( \left| a^\eps \right|^{2},
\left|\ta\right|^{2}) =
  G(r_1,r_2)B(r_1,r_2)\big|_{(r_1,r_2)=( \left| a^\eps
    \right|^{2},\left|\ta\right|^{2}) }, 
\end{equation*}
where $\tbeta$ satisfies an equation of the form 
\begin{equation}\label{eq:betaasdesired}
\partial_{t}\tbeta+ L(\ta,\tphi,\partial_{x})\tbeta+\tg \DIV \(
\eps\IM(\overline{a}^\eps\nabla a^\eps)\)=0, 
\end{equation}
and $L$ is a first order differential operator. 
Proposition~\ref{prop:betag} below shows that it is possible to do so. 
Before giving this precise statement, we introduce convenient
notations, and explain  how to formally find $\tbeta$. 

\smallbreak

Introduce the position densities
\begin{equation*}
    \tp\defn |a|^2 \in C^\infty([0,T^*[\times \R^n)\quad ;\quad \trho\defn
    |a^\eps|^2 =\la \tu\ra^2  
    \in C^\infty([0,T^\eps[\times\R^n).
\end{equation*}
Let $\tv = \nabla \phi$. Elementary computations show
that:
\begin{align}
&\partial_{t}\tp + \DIV (\tp\tv)=0,\label{eq:r1}\\
&\partial_{t}\trho + \DIV \IM\(\eps \overline{u}^\eps\nabla \tu\)
=0,\label{eq:r2}\\ 
&\partial_{t}\trho + \DIV \(\IM ( \eps\overline{a}^\eps\nabla a^\eps) +\trho
\tv\)=0.\label{eq:r3} 
\end{align}
Denote
$$
\tJ\defn \eps \IM (\overline a^\eps\nabla  a^\eps).
$$
By writing
$$
\partial_{t}\tbeta = (\partial_{r_{1}}B)(\trho,\tp)\partial_{t}\trho
+(\partial_{r_{2}}B)(\trho,\tp)\partial_{t}\tp,
$$
we compute, from \eqref{eq:r1} and \eqref{eq:r3}:
$$
\partial_{t}\tbeta + (\partial_{r_{1}}B)(\trho,\tp)\DIV(\tJ+\trho \tv) 
+ (\partial_{r_{2}}B)(\trho,\tp)\DIV(\tp \tv)=0.
$$
Hence, in order to have an equation of the desired form
\eqref{eq:betaasdesired}, we impose 
$$
\partial_{r_{1}}B(r_{1},r_{2}) = G(r_{1},r_{2}).
$$
Since on the other hand, 
$$
G(r_{1},r_{2})B(r_{1},r_{2}) = r_{1}^\sigma-r_{2}^\sigma, 
$$
this suggests to choose $\tbeta$ such that
\begin{equation}\label{B0}
(\tbeta)^2 =
\frac{2}{\sigma+1}(\trho)^{\sigma+1}-2\tp^\sigma\trho+f(\tp). 
\end{equation}
To obtain an operator $L$ which is linear with respect to $\tbeta$ we choose
\begin{equation}\label{B1}
(\tbeta)^2 =
\frac{2}{\sigma+1}(\trho)^{\sigma+1}-\frac{2}{\sigma+1}\tp^{\sigma+1} 
-2\tp^\sigma(\trho-\tp).
\end{equation}
With this choice, we formally compute:
\begin{equation*}
\partial_{t}\beta^{\eps} +\eps \tg\DIV( \IM(\overline{ a}^\eps\nabla
a^\eps) )  
+\tv\cdot\nabla\tbeta + \frac{\sigma+1}{2}\tbeta\DIV \tv =0.
\end{equation*}
Before deriving this equation rigorously, examine
the right hand side of \eqref{B1}. Taylor's formula yields 
$$
\frac{2}{\sigma+1}(\trho)^{\sigma+1}-\frac{2}{\sigma+1}\tp^{\sigma+1}
-2 \tp^\sigma(\trho-\tp) = (\trho-\tp)^2 Q_{\sigma}(\trho,\tp),
$$
where $Q_{\sigma}$ is given by:
\begin{equation}\label{eq:defQ}
  Q_{\sigma}(r_{1},r_{2})\defn
2\si \int_0^1 (1-s) \( r_2 +s (r_1-r_2)\)^{\si -1}
ds. 
\end{equation}
Note that there exists $C_\si$ such that:
\begin{equation}\label{eq:Qpardessous}
  Q_{\sigma}(r_{1},r_{2}) \ge C_\si \( r_1^{\si -1} + r_2^{\si -1}\). 
\end{equation}
\begin{notation}\label{nota:GBP}
Let $\sigma\in\N$. Introduce
$$
G_\si(r_1,r_2)=\frac{P_{\sigma}(r_1,r_2)}{\sqrt{Q_{\sigma}(r_1,r_2)}}\quad
;  \quad 
B_\si(r_1,r_2)\defn (r_1-r_2) \sqrt{Q_{\sigma}(r_1,r_2)},
$$
where $Q_\si$ is given by \eqref{eq:defQ} and
$$
P_{\sigma}(r_{1},r_{2})=\frac{r_{1}^\si-r_{2}^\si}{r_{1}-r_{2}}
=\sum_{\ell=0}^{\sigma-1}r_{1}^{\sigma-1-\ell}r_{2}^{\ell}. 
$$
\end{notation}
\begin{example}
For $\sigma=1,2,3$, we compute
\begin{alignat*}{2}
&G_{1}=1, \qquad  &&  B_{1}=r_{1}-r_{2}.\\ 
&G_{2}=\sqrt{\frac{3}{2}}\frac{r_{1}+r_{2}}{\sqrt{r_{1}+2r_{2}}}\virgp
\qquad && 
B_{2}=\sqrt{\frac{2}{3}}(r_{1}^2-r_{2}^2)\sqrt{r_{1}+2r_{2}}.\\
&G_{3}=\sqrt 2\frac{r_{1}^2+r_{1}r_{2}+r_{2}^2}{\sqrt{(r_{1}-r_{2})^2+
2r_{2}^2}}\virgp 
\qquad 
&& B_{3}=\frac{1}{\sqrt 2}(r_{1}^2-r_{2}^2)\sqrt{(r_{1}-r_{2})^2+2r_{2}^2}.
\end{alignat*}
\end{example}

A remarkable fact is that, although the functions $G_{\si}$ and
$B_\si$ are not smooth for $\sigma\ge 2$,  
one can compute an evolution equation for the unknown 
$\tbeta\defn B_\si\bigl(\la a^\eps\ra^2,\la\ta\ra^2\bigr)$. We have
the following key proposition. 
\begin{proposition}\label{prop:betag}
With $G_{\si}$ and $B_{\si}$ as above, define 
$$
\tbeta\defn B_\si\bigl(\la a^\eps\ra^2,\la\ta\ra^2\bigr),\qquad  
\tg\defn G_\si\bigl(\la a^\eps\ra^2,\la\ta\ra^2\bigr).
$$
Then $\tbeta\in
C^1([0,\tau^\eps[\times\R^n)$ and $\tg \in
C^0([0,\tau^\eps[\times\R^n)$, where $\tau^\eps=\min
(T^*,T^\eps)$. Moreover, 
\begin{equation}\label{eq:12h25}
\partial_{t}\tbeta +\eps \tg\DIV( \IM(\overline{ a}^\eps\nabla a^\eps) ) 
+\tv\cdot\nabla\tbeta + \frac{\sigma+1}{2}\tbeta\DIV \tv =0.
\end{equation}
\end{proposition}
\begin{remark}
Again, note the dichotomy between $\sigma=1$ and $\sigma\ge 2$. 
If $\sigma=1$ then, by definition, $\tg=1$ and $\tbeta=\rho^\eps-\rho$
are $C^\infty$ functions. Moreover  
\eqref{eq:12h25} simply reads
\begin{equation*}
\partial_{t}\tbeta +\eps \DIV( \IM(\overline{ a}^\eps\nabla a^\eps) ) 
+\DIV (\tv\tbeta)=0,
\end{equation*}
corresponding to the equation for $\tq=\eps^{-1}\tbeta$ in
Section~\ref{sec:cubic},  
and which follows directly by subtracting \eqref{eq:r1} from
\eqref{eq:r3}. 
\end{remark}
\begin{proof}
The regularity properties of $\tbeta$ and $\tg$ follow from
Lemmas~\ref{lemI1} and~\ref{lemma:MUK}, along with the definition of $\tbeta$
and $\tg$ (see Notation~\ref{nota:GBP}, and
\eqref{eq:Qpardessous}).  
\smallbreak

Since by definition
\begin{align*}
 &\tbeta(\partial_{r_{1}}B_\si)(\trho,\tp)=(\trho)^\sigma-\tp^\sigma, \\
&\tbeta(\partial_{r_{2}}B_\si)(\trho,\tp)=\sigma
(\rho^\sigma-\tp^{\sigma-1}\rho^\eps), 
\end{align*}
we have
\begin{align*}
&\tbeta\partial_{t}\tbeta \\ 
&= \tbeta(\partial_{r_{1}}B_\si)(\trho,\tp)\partial_{t}\trho
+\tbeta(\partial_{r_{2}}B_\si)(\trho,\tp)\partial_{t}\tp\\
&=-\tbeta(\partial_{r_{1}}B_\si)(\trho,\tp)\DIV(J_{\eps}+\trho \tv) 
-\tbeta(\partial_{r_{2}}B_\si)(\trho,\tp)\DIV(\tp \tv)\\
&=-((\trho)^\sigma-\tp^\sigma)\DIV(J_{\eps}+\trho \tv)
-\sigma (\rho^\sigma-\tp^{\sigma-1}\rho^\eps)\DIV(\tp \tv).
\end{align*}
From this we compute
$$
\tbeta\Bigl(\partial_{t}\beta^{\eps} +\eps g^\eps\DIV(
\IM(\overline{ a}^\eps\nabla a^\eps) )  
+\tv\cdot\nabla\tbeta + \frac{\sigma+1}{2}\tbeta\DIV \tv\Bigr) =0.
$$
Introduce
\begin{align*}
\tomega  \defn &\{ \trho=\tp\}=\bigl\{ (t,x)\in
[0,\tau^\eps[\times\R^n\, \arrowvert \, 
\trho(t,x) = \tp(t,x) \,\bigr\} \\
= &\([0,\tau^\eps[\times\R^n\)\setminus \{
\tbeta\neq 0\}
\qquad (\text{by }\eqref{eq:Qpardessous}).
\end{align*}
Then \eqref{eq:12h25} holds on $
\([0,\tau^\eps[\times\R^n\)\setminus 
\tomega$, and hence on $\overline{\([0,\tau^\eps[\times\R^n\)\setminus 
\tomega}$ by continuity. To prove the proposition, it thus suffices to show
$$
\partial_{t}\beta^{\eps} +\eps g^\eps\DIV( \IM(\overline 
  a^\eps\nabla a^\eps) )  
+\tv\cdot\nabla\tbeta + \frac{\sigma+1}{2}\tbeta\DIV \tv
=0\quad\text{on}\quad \itomega ,
$$
where $\overset{\circ}{A}$ denotes the interior of the set $A$. 
Since $\tbeta=0$ on $\itomega$, it is enough to prove 
that $\DIV( \IM(\overline a^\eps\nabla a^\eps))=0$ on $\itomega$. This in
turn follows from  \eqref{eq:r1} and \eqref{eq:r3}, which yield:
$$
\DIV( \IM(\overline a^\eps\nabla a^\eps)) = -\eps^{-1}\Bigl(
\partial_{t}(\trho-\tp) +\DIV ( (\trho-\tp)\tv)\Bigr).
$$
This completes the proof.
\end{proof}

We will see that
$\left|  a^\eps \right|^{2\si} - \left|\ta\right|^{2\si}$ is of order 
$\O(\eps)$, so we naturally set
\begin{equation*}
  \tw\defn \nabla a^\eps \quad ;\quad\tq\defn \eps^{-1}\tbeta.
\end{equation*}
We infer from the previous computations that $(a^\eps, \tw,\tq)$
solves: 
\begin{equation}\label{eq:systprefinal}
\left\{
\begin{aligned}
&  \partial_t  a^\eps +\tv\cdot\nabla  a^\eps
+\frac{1}{2} a^\eps \DIV \tv
- i\frac{\eps}{2}\Delta  a^\eps
=-i \tg \tq  a^\eps . \\
&\partial_t \tw +\tv\cdot\nabla \tw
+\frac{1}{2}\tw\DIV \tv +
\tw\cdot\nabla \tv+\frac{1}{2} a^\eps\nabla\DIV \tv -
i\frac{\eps}{2}\Delta \tw =\\ 
&\qquad = -i  \tq \nabla \(a^\eps \tg\) -i a^\eps \tg \nabla \tq,\\ 
&\partial_{t}\tq +\tv\cdot\nabla \tq + \tg\DIV( \IM(\overline a^\eps\tw) ) 
+ \frac{\sigma+1}{2} \tq\DIV \tv =0.
\end{aligned}
\right.
\end{equation}
Simply by writing 
$$
\tg\DIV( \IM(\overline a^\eps\tw) ) = \IM(\tg \overline a^\eps\DIV \tw),
$$
we can rewrite the previous system as
\begin{equation}\label{eq:systfinal}
\left\{
\begin{aligned}
&  \partial_t  a^\eps +\tv\cdot\nabla  a^\eps
- i\frac{\eps}{2}\Delta  a^\eps
=-\frac{1}{2} a^\eps \DIV \tv-i \tg \tq  a^\eps . \\
&\partial_t \tw +\tv\cdot\nabla \tw
+i a^\eps \tg \nabla \tq
-i\frac{\eps}{2}\Delta \tw =\\ 
&\quad = -\frac{1}{2}\tw\DIV \tv -\tw\cdot\nabla \tv-\frac{1}{2}
a^\eps\nabla\DIV \tv -i  \tq 
 \nabla \(a^\eps \tg\) ,\\  
&\partial_{t}\tq +\tv\cdot\nabla \tq + \IM(\tg \overline a^\eps\DIV \tw)
=- \frac{\sigma+1}{2} \tq\DIV \tv.
\end{aligned}
\right.
\end{equation}
Note that in view of Assumption~\ref{assu:general}, 
\begin{equation}\label{eq:CIbornees}
 \lA a^\eps_{\mid t=0}\rA_{H^s(\R^n)} +\lA \tw_{\mid
 t=0}\rA_{H^s(\R^n)}= \O(1),\
 \forall s\ge 0.  
\end{equation}
A similar estimate for the initial data of $q^\eps$ is a more delicate
issue, since $B_\si$ is not a smooth function. We postpone this
estimate to \S\ref{sec:quasilin}.
 
The left hand side of \eqref{eq:systfinal} is a first order
quasi-linear symmetric hyperbolic 
system, plus a second order skew-symmetric term. The right hand side
can be viewed as a semi-linear source term.
We deduce from Proposition~\ref{prop:betag}: 
\begin{corollary}\label{coro:betag}
On $[0,\tau^\eps[\times \R^n$, the function 
$U^\eps\defn (2\tq,a^\eps,\overline a^\eps,\tw,\overline \psi^\eps)$
satisfies an equation of the form 
\begin{equation*}
\partial_{t} U^\eps + \sum_{1\le j\le n}A_{j}(\tv, a^\eps
\tg,\overline a^\eps \tg)\partial_{j}U^\eps +  
\eps L(\partial_{x})U^\eps= E(\tPhi,U^\eps,a^\eps \tg,\nabla(
a^\eps \tg)), 
\end{equation*}
where $\tPhi=(\nabla\tphi,\nabla^2\tphi,\nabla^3 \tphi)$, 
the $A_{j}$'s are Hermitian matrices 
linear in their arguments, $\mathcal{L}(\partial_{x})=\sum
L_{jk}\partial_j\partial_k$  
is a skew-symmetric second order differential operator with constant
coefficients,  
and $E$ is a $C^\infty$ function of its arguments, vanishing at the origin. 
\end{corollary}

Theorem~\ref{theo:main} and Proposition~\ref{prop:quad} are
consequences of the following:
\begin{theorem}\label{theo:tout}
  Let $n\le 3$, and let Assumption~\ref{assu:general} be satisfied. 
There exists $T\in ]0,T^*[$, where $T^*$ is given by
Lemma~$\ref{weaklemma:MUK}$, such that the following holds. 
For all $\eps\in]0,1]$, the Cauchy problem \eqref{eq:nlssemi} has a unique 
solution $\tu\in C([0,T];H^\infty(\R^n))$. Moreover,
\begin{equation}\label{eq:uniform2}
\sup_{\eps \in ]0,1]}\(\bigl\lVert a^\eps
\bigr\rVert_{L^\infty([0,T];H^k(\R^n))} +
\bigl\lVert q^\eps
 \bigr\rVert_{L^\infty([0,T];H^{k-1}(\R^n))}\) <+\infty, 
\end{equation}
where the index $k$ is as follows:
\begin{itemize}
\item If $\si =1$, then $k\in \N$ is arbitrary.
\item If $\si = 2$ and $n=1$, then we can take $k=2$.
\item If $\si = 2$ and $2\le n\le 3$, then we can take $k=1$.
\item If $\si\ge 3$, then we can take $k=\si$. 
\end{itemize}
\end{theorem}
Considering the first part of the estimate \eqref{eq:uniform2} yields
Theorem~\ref{theo:main}. We 
explain why Proposition~\ref{prop:quad} is a consequence of
\eqref{eq:uniform2} in \S\ref{sec:quad}. 

\subsection{Quasi-linear analysis}
\label{sec:quasilin}
We now have to estimate $( a^\eps,\tq,\tw)$ in Sobolev spaces. 
Let us briefly explain the difficulty. 
To clarify matters, suppose that $g^\eps=G(\la a^\eps\ra^2,\la
a\ra^2)$ for some smooth 
function $G\in C^\infty(\R^2)$.  
In particular this is so in the cubic case $\si=1$. 
Then, in view of Corollary~\ref{coro:betag}, 
$U^\eps\defn (2\tq, a^\eps,\overline a^\eps,\tw,\overline\tw)\in
C([0,\tau^\eps[;H^\infty(\R^n))^{3+2n}$  
satisfies a system of the form~\eqref{system:hyp},
$$
\partial_{t}U^\eps+ \sum_{1\le j\le
  n}\mathcal{A}_{j}(\tPhi,U^\eps)\partial_{j}U^\eps 
+\eps L(\partial_{x})U^\eps=\mathcal{E}(\tPhi,U^\eps), 
$$
where $\tPhi\defn
(\la a\ra^2,\nabla\tp,\nabla\tphi,\nabla^2\tphi,\nabla^3 \tphi)$. The key  
difference  
with the system in Corollary~\ref{coro:betag} is the absence of
dependence upon the extra unknown~$\tg$.   
Then, Lemma~\ref{lemma:hyp} yields estimates in Sobolev spaces (of
arbitrary order).  

\smallbreak

Assume now $\si \ge 2$. One can check that the previous symmetrization
  provides us with uniform \emph{a 
  priori} estimates in $L^2$. However, the estimates of the derivatives
require a careful analysis. 
Indeed, recall that
\begin{equation*}
g^\eps=G_{\si}(\la a^\eps\ra^2,\la a\ra^2)\quad\text{with}\quad
  G_\si(r_1,r_2 )= 
  \frac{P_{\sigma}\(r_1,r_2\)}{\sqrt{Q_{\sigma}\(r_1,r_2\)}} \virgp
\end{equation*}
where $P_\si$ and $Q_\si$ are defined in Notation~\ref{nota:GBP} and
\eqref{eq:defQ} respectively. Therefore, $G_\si$ need not be smooth at
the origin. The classical approach, which consists in
differentiating  
the equations, thus certainly fails here. 
Yet, as we will see, we need only estimate $a^\eps g^\eps$ in $H^{\si}$. 
Introduce 
\begin{equation}\label{eq:Fsi}
F_\si(z,z')=z G_\si\(\la z\ra^2,\la z'\ra^2\)
:\qquad  a^\eps g^\eps =F_\si\(a^\eps,\ta\).   
\end{equation}
One can check that $F_{\si}\in C^{\si-1}$ but $F_{\si}\not\in C^{\si}$. 
Hence, to estimate $a^\eps g^\eps$ in $H^{\si}$, 
one cannot use the usual nonlinear estimates. 
Instead, we will use that $F_{\si}$ is homogeneous of degree $\si$ and
the following lemma. 

\begin{lemma}\label{lemm:comp}
Let $p\ge 1$ and $m\ge 2$ be integers and consider 
$F\colon \R^p\rightarrow \C$.  Assume that 
$F\in C^\infty(\R^p\setminus \{0\})$ is homogeneous of degree $m$, that is:
$$
F(\lambda y)= \lambda^m F(y),\qquad \forall \lambda \ge 0,\forall y\in \R^p.
$$
Then, for $n\le 3$, there exists $K>0$ such that, for all 
$u\in H^m(\R^n)$ with values in $\R^p$, 
$F(u)\in H^m(\R^n)$ and
$$
\lA F(u)\rA_{H^m}\le K \lA u\rA_{H^{m}}^m.
$$
The same is true when $m=1$ and $n\in\N$. 
\end{lemma}
\begin{remark}
Note that the result is false for $n\ge 4$ and $m\ge 2$. 
Also, one must not expect $F(u)\in H^{m+1}(\R^n)$, even for $u\in
H^\infty(\R^n)$.  
For instance, if
$$
n=1=p,\quad m=2,\quad F(y)=y\la y\ra,\quad u(x)=xe^{-x^2},
$$
then 
$F(u)\in H^2(\R)$ and 
$F(u)\not\in H^3(\R)$. Similarly, in general, one 
  must not expect $F_\si(u,v)\in H^{\si+1}(\R^n)$, even for 
  $(u,v)\in H^\infty(\R^n)^2$. 
\end{remark}

\begin{proof}
We prove the result by induction on $m$. 
Consider first the case $m=2$. Observe that, by assumption, $F\in
C^{m-1}(\R^p)$.  
To regularize $F$, let $\chi\in C^\infty_{0}(\R^p)$ 
be such that $0\le\chi\le 1$, $\chi(y)=1$ for $|y|\le 1$ and 
$\chi(y)=0$ for $|y|\ge 3$, with $|\nabla \chi(y)|\le 1$. For $\ell\in
\N$, define  
$F_{\ell}\in C^\infty(\R^p)$ by
$$
F_{\ell}(y)=\(1-\chi\(\ell y\)\)F(y).
$$
We claim that, for all $y\in\R^p$ and all $\ell\in\N$,
\begin{align*}
\la F_{\ell}(y)\ra \le C_{F} |y|^2, \quad
\la \partial_{j} F_{\ell}(y)\ra\le 4C_{F} \la y\ra,\quad
\la \partial_{j}\partial_{k} F_{\ell}(y)\ra\le 4 C_{F},
\end{align*}
where $\partial_{j}=\partial_{y_{j}}$ and
$$
C_{F}\defn \sup_{|z|\le 3} 
\la  F (z)\ra + \sup_{1\le j\le p}\sup_{|z|\le 3} 
\la \partial_{j} F (z)\ra + 
\sup_{1\le j,k\le p}\sup_{|z| = 1} 
\la \partial_{j}\partial_{k} F (z)\ra.
$$
Since $F_{\ell}$ vanishes in a neighborhood of the origin, it suffices 
to establish these bounds for $y\neq 0$. The first bound follows from 
the homogeneity: $| F_{\ell}(y)|\le |F(y)|= |y|^2 |F(y/|y|)|$. For 
the second one, compute
\begin{align*}
\partial_{j} F_{\ell} (y)
=(1-\chi(\ell y)) \partial_{j} F(y)-\ell^{-1}(\partial_{j} \chi)(\ell
y) F(\ell y), 
\end{align*}
where we used $\ell F(y)=\ell^{-1}F(\ell y)$. 
Since $1\le |\ell y|\le 3$ on the support 
of $(\partial_{j} \chi )(\ell y)$, and since  $\partial_{j} F\colon
\R^p\rightarrow \C$  
is homogeneous of degree~$1$, we infer
$$
|\partial_{j} F_{\ell} (y)|\le |y| 
\( \sup_{|z|\le 3}\la\partial_{j} F\(z\)\ra+ 3 \sup_{z\in\R^p}
\la (\partial_{j}\chi)(z)F(z) \ra\) \le 4C_{F} |y|.
$$
The same raisoning yields
\begin{align*}
\partial_{j}\partial_{k} F_{\ell} (y)&=
(1-\chi(\ell y))\partial_{j}\partial_{k} F(y)
-(\partial_{j} \chi)(\ell y) \partial_{k} F(\ell y)\\
&\quad-(\partial_{k} \chi)(\ell y) \partial_{j} F(\ell y)
-(\partial_{j}\partial_{k}\chi)(\ell y)F(\ell y).
\end{align*}
The last three terms are clearly bounded by $C_{F}$ since 
$|\ell y|\le 3$ on the support of $\chi(\ell y)$. Also, the first term
is bounded by  
$C_{F}$ since 
$\partial_{j}\partial_{k}F\colon \R^p\setminus \{0\}\rightarrow \C$
is homogeneous of  
degree $0$. This completes the proof of the claim.

\smallbreak
With these preliminary established, we easily obtain that there exists 
$K$ such that for all $\ell\in\N$ and all $u\in H^2(\R^n)$ with values
in $\R^p$, 
\begin{align*}
&\lA F_{\ell}(u)\rA_{L^2}\le K \lA u\rA_{L^\infty}\lA u\rA_{L^2},\\
&\lA \nabla  F_{\ell}(u)\rA_{L^2}\le K\lA u\rA_{L^\infty}\lA
\nabla u\rA_{L^2},\\
&\lA \nabla^2 F_{\ell}(u)\rA_{L^2}\le K\lA u\rA_{L^\infty}
\lA \nabla^2 u\rA_{L^2}+K\lA
\nabla u\rA_{L^4}^2.  
\end{align*}
The Sobolev embeddings $H^1(\R^n)\subset L^6(\R^n)$ and 
$H^2(\R^n)\subset L^\infty(\R^n)$ for $n\in \{1,2,3\}$ then imply that
there exists  
a constant $K$ such that, for all $\ell\in\N$ and all $u\in H^2(\R^n)$,
$$
\lA F_{\ell}(u)\rA_{H^2}\le K \lA u\rA_{H^{2}}^2.
$$
This in turn implies the desired result for $F(u)$ by using the
dominated convergence theorem and a duality argument.  
Indeed, for all $\varphi\in C^\infty_{0}(\R^n)$, 
\begin{align*}
\la \int F(u)\varphi \, dx\ra &= 
\la \lim_{\ell\rightarrow +\infty} \int F_{\ell}(u)\varphi\, dx\ra
\le \limsup_{\ell\rightarrow +\infty}
\lA F_{\ell}(u)\rA_{H^2}\lA \varphi\rA_{H^{-2}}\\
&\le K \lA u\rA_{H^{2}}^2 
\lA \varphi\rA_{H^{-2}},
\end{align*}
which implies $F(u)\in H^2(\R^n)$ together with 
$\lA F(u)\rA_{H^2}\le K \lA u\rA_{H^{2}}^2$.

\smallbreak
Assume now the result at order $m\ge 2$, and prove the result at order 
$m+1$. Let $F\in C^\infty(\R^p\setminus \{0\})$ be homogeneous of
degree $m+1$. We have
\begin{equation*}
  \lA F(u)\rA_{L^2}\le K \lA u\rA_{L^\infty}^{m}\lA u\rA_{L^2}\les
  \lA u\rA_{H^{m+1}}^{m+1}. 
\end{equation*}
Since $m>3/2\ge n/2$, $H^m(\R^n)$ is an algebra and
\begin{equation*}
  \lA \nabla F(u)\rA_{H^m}\le K \lA \nabla u\rA_{H^m}\lA F'(u)\rA_{H^m}.
\end{equation*}
By assumption, $F'\in C^\infty(\R^p\setminus \{0\})$ is homogeneous of
degree~$m$, hence the induction assumption yields: 
\begin{equation*}
  \lA F'(u)\rA_{H^m}\le K \lA u\rA_{H^m}^m.
\end{equation*}
Therefore,
\begin{equation*}
  \lA \nabla F(u)\rA_{H^m}\le K \lA u\rA_{H^{m+1}}^{m+1}.
\end{equation*}

The case $m=1$ can be treated in a similar fashion. 
\end{proof}
The lemma turns out to be useful to estimate the source term in
\eqref{eq:systfinal}, but also to estimate the initial data for
$q^\eps$. By definition, we have
\begin{equation*}
  q^\eps = \frac{|z|^2-|z'|^2}{\eps}\mathcal{Q}_\si
  (z,z')\big|_{(z,z')=(a^\eps,a)}, 
\end{equation*}
where 
\begin{align*}
\mathcal{Q}_\si (z,z') &=\sqrt{Q_\si\( |z|^2,|z'|^2\)} \\
&= \(2\si \int_0^1 (1-s)\( |z'|^2
+s (|z|^2-|z'|^2)\)^{\si-1}ds\)^{1/2}.
\end{align*}
The function $\mathcal{Q}_\si$ is not smooth, but homogeneous of degree
$\si -1$. So when $\si\ge 3$, we can estimate $q^\eps$ in $H^{\si-1}$
at time $t=0$ thanks to this lemma. See \S\ref{sec:si>2}.
\smallbreak

To complete the proof of
  Theorem~\ref{theo:tout}, in view of Lemma~\ref{lemI1}, 
we seek an $H^2$ estimate of $ a^\eps$, since
  \begin{equation*}
    H^2(\R^n)\subset L^\infty(\R^n),\quad n\le 3. 
  \end{equation*}
This boils down to an $H^1$ estimate of $U^\eps$ defined in
Corollary~\ref{coro:betag}.  However, the estimates of the derivatives
require a   
careful analysis. Indeed, the classical approach, which consists in
differentiating  
the equations, certainly fails here because $G_\si$ need not be
smooth. Moreover, Lemma~\ref{lemma:hyp} requires to control $U^\eps$ in
$H^s$ with $s>n/2+1$, so we would demand $s=3$ for $n=3$ and $s\in
\N$. In view of Lemma~\ref{lemma:hyp} and Corollary~\ref{coro:betag}, we
have to estimate $a^\eps g^\eps$ in $H^4$. Because of the lack of
smoothness of $G_\si$, such an estimate seems hopeless in general. We
therefore proceed in two steps. First, using the particular structure
exhibited in Corollary~\ref{coro:betag}, we relax the assumption
$s>n/2+1$ in Lemma~\ref{lemma:hyp}, to $s>n/2$. Next, we use
Lemma~\ref{lemm:comp} to overcome the lack of smoothness of 
$G_\si$, and obtain the desired \emph{a priori} estimates.

\begin{proposition}\label{prop:quasilin}
  Assume $\si \ge 2$. Let $U^\eps$ be the vector-valued function given by
  Corollary~\ref{coro:betag}, and $m>n/2$. Then for all for
  $t\in[0,\tau^\eps[$, it satisfies 
  the following \emph{a priori} estimate:
  \begin{equation*}
    \sup_{s\in [0,t]}\lA U^\eps(s)\rA_{H^m}\le \lA U^\eps(0)\rA_{H^m}
    e^{t C(M^\eps(t))}, 
\end{equation*} 
with $M^\eps(t)\defn  \lA \tPhi 
\rA_{L^\infty([0,t];H^{m})}+ \lA U^\eps
\rA_{L^\infty([0,t];H^{m})}+ \lA  a^\eps \tg
\rA_{L^\infty([0,t];H^{m+1})}$.
\end{proposition}
\begin{proof}[Sketch of the proof] Resume the proof of
  Lemma~\ref{lemma:hyp}. The quantities that appear in $M^\eps$ are
  those on the last three lines of the proof of Lemma~\ref{lemma:hyp}.
First, we have:
  \begin{align*}
    \lA \nabla A_{j}(\tv, a^\eps
\tg,\overline a^\eps \tg)\rA_{L^{\infty}} &\le C\(\lA \tv
\rA_{W^{1,\infty}}+ \lA a^\eps \tg
\rA_{W^{1,\infty}} \)\\
& \le C\( \lA \tv \rA_{H^{m+1}}+ \lA a^\eps \tg\rA_{H^{m+1}}\). 
  \end{align*}
Next, since $A_j$ is linear in its arguments, $\widetilde A_j=A_j$
and:
\begin{align*}
\lA [A_{j},\Lambda^{m}]\partial_{j}U^\eps\rA_{L^{2}} 
&\le K \norm{{A}_{j}}_{H^{m+1}}
\lA U^\eps\rA_{H^{m}}\\
&\le C\(\lA \tv \rA_{H^{m+1}}+ \lA
a^\eps \tg\rA_{H^{m+1}}\) \lA U^\eps\rA_{H^{m}}.
\end{align*} 
Finally,
\begin{align*}
\lA E(\tPhi,U^\eps,a^\eps \tg ,\nabla(a^\eps \tg)) \rA_{H^m}
\le C\(\lA \tPhi\rA_{H^m} ,\lA U^\eps\rA_{H^m}, \lA a^\eps
\tg\rA_{H^{m+1}}\). 
\end{align*}
We conclude the proof thanks to Gronwall lemma. 
\end{proof}

\subsection{The case $\si \ge 3$}
\label{sec:si>2}

Recall that from \eqref{eq:Fsi}, 
\begin{equation*}
  a^\eps g^\eps =F_\si\(a^\eps,\ta\),
\end{equation*}
where $F_\si$ is homogeneous of degree $\si$. 
For $\si \ge 3$ and $n\le 3$, Lemma~\ref{lemm:comp} yields
\begin{equation*}
  \|a^\eps \tg \|_{H^\si} \le K\( \|a^\eps\|_{H^\si} +
  \|\ta\|_{H^\si}\)^\si.
\end{equation*}
Proposition~\ref{prop:quasilin} with 
$m=\si-1\ge 2>n/2$ shows that there exists a function~$C$ 
from $[0,+\infty[$ to $[0,+\infty[$ such that, for all 
$\eps\in ]0,1]$ and all $t\in [0,\tau^\eps[$,
$$
\lA U^\eps(t)\rA_{H^{\sigma-1}}
\le \lA U^\eps(0)\rA_{H^{\si-1}}\exp (tC(M^\eps(t))),
$$
where 
$$
M^\eps(t)\defn \lA U^\eps\rA_{L^\infty([0,t];H^{\sigma-1}(\R^n))}
+\lA (a,\phi)\rA_{L^\infty([0,t];H^{\sigma+2}(\R^n))}.
$$
It remains to estimate the initial data. By definition, we have
$$
\lA U^\eps(0)\rA_{H^{\si-1}}
\les \lA q^\eps(0)\rA_{H^{\si-1}}+\lA a^\eps(0)\rA_{H^{\si}}.
$$
The second term is uniformly bounded by assumption. To estimate 
the first term, recall that
\begin{equation*}
  q^\eps = \frac{|z|^2-|z'|^2}{\eps}\mathcal{Q}_\si
  (z,z')\big|_{(z,z')=(a^\eps,a)}, 
\end{equation*}
where 
\begin{align*}
\mathcal{Q}_\si (z,z') &=\sqrt{Q_\si\( |z|^2,|z'|^2\)} \\
&= \(2\si \int_0^1 (1-s)\( |z'|^2
+s (|z|^2-|z'|^2)\)^{\si-1}ds\)^{1/2}.
\end{align*}
The function $\mathcal{Q}_\si$ is not smooth, but homogeneous of degree
$\si -1$. To estimate  
$q^\eps$ at time $t=0$, we use 
the usual product rule in Sobolev space and Lemma~\ref{lemm:comp} (applied 
with $F(y_{1},\ldots,y_{4})=
\mathcal{Q}_{\si}(y_{1}+iy_{2},y_{3}+iy_{4})$):  if $\si\ge 3$, with
$m=\si-1\ge 2$, we obtain  
\begin{align*}
\lA q^\eps(0)\rA_{H^{\si-1}}
&\les\lA \eps^{-1}\(\la a^\eps(0)\ra^2 -\la a(0)\ra^2\)\rA_{H^{\si-1}}
\lA \mathcal{Q}_{\si}\( a^\eps(0), a(0)\)\rA_{H^{\si-1}}\\
&\les \lA \eps^{-1}\(\la a^\eps(0)\ra^2 -\la a(0)\ra^2\)\rA_{H^{\si-1}}
\lA \(a^\eps(0),a(0)\)\rA_{H^{\si-1}}^{\si-1}.
\end{align*}
The assumption $a^\eps_{0}-a_{0}=\mathcal{O}(\eps)$ in $H^s$ for all
$s>0$ then implies 
\begin{equation}
  \label{eq:CIq}
\sup_{\eps\in ]0,1]}\lA
 q^\eps(0)\rA_{H^{\sigma-1}}<+\infty, 
\end{equation}
hence
$$
\sup_{\eps\in ]0,1]}\lA U^\eps(0)\rA_{H^{\sigma-1}}<+\infty.
$$
Consequently, since 
$\lA u^\eps e^{-i\phi/\eps} \rA_{H^{\sigma}}=\lA a^\eps\rA_{H^{\sigma}}
\le \lA U^\eps\rA_{H^{\sigma-1}}$, 
the same continuity argument as in Section~\ref{sec:cubic} 
completes the proof of Theorem~\ref{theo:tout} 
in the case $\si\ge 3$. 
\subsection{The case $\si =2$}
\label{sec:si=2}
For $\si =2$, we have $m=\si-1>n/2$ only when $n=1$. 
The last point in
Lemma~\ref{lemm:comp} shows that 
\begin{equation*}
\sup_{\eps\in ]0,1]}\lA
 q^\eps(0)\rA_{H^{1}(\R)}<+\infty. 
\end{equation*}
We can then proceed as in the case $\si\ge 3$, to prove the second case in
Theorem~\ref{theo:tout}.  
\smallbreak

Finally, when $\si=2$ and $2\le n\le 3$, recall that we already know
that for fixed $\eps \in ]0,1]$, $u^\eps$ is global in
time,  $u^\eps \in C(\R,H^1)$. For $n=2$, this is so since every
defocusing, homogeneous nonlinearity is $H^1$ sub-critical. For $n=3$,
the nonlinearity is $H^1$ critical, and this property follows from
\cite{CKSTTAnnals}. 
The proof of the estimate is based 
on an interesting feature  
of the equation for $\tbeta$ (see Proposition~\ref{prop:betag}), which
does not appear in Corollary~\ref{coro:betag}.  
In the introduction, we claimed that the previous nonlinear
symmetrization of the equations  
implies a local version of the modulated energy estimate. 
To see this, introduce
$$
e^\eps  
\defn \la a^\eps\ra^2+\la \tw\ra^2 + \la q^\eps \ra^2\in
C^1([0,\tau^\eps[\times\R^n). 
$$
It satisfies an equation of the form $\partial_{t}e^\eps +
\DIV( \eta^\eps)+ \flat^\eps =\mathcal{O}(e^\eps)$, where $\int
\flat^\eps=0$.   
Indeed, directly from \eqref{eq:systprefinal}, we compute
\begin{equation*}
\begin{aligned}
&\partial_{t}e^\eps 
+ \DIV( v e^\eps) + 2\DIV\bigl(\IM (\tg \tq \overline a^\eps \tw
)\bigr) +\eps \IM \(\overline a^\eps \Delta a^\eps +\overline {\psi^\eps}
\Delta \psi^\eps\)   \\
&\quad=-\sigma\la \tq\ra^2\DIV \tv 
-\RE\(( 2\tw\cdot\nabla \tv+a^\eps\nabla\DIV \tv)\overline{\tw}\).
\end{aligned}
\end{equation*}
Hence we have obtained 
an evolution equation for a modulated energy, which yields the desired 
modulated energy estimate. Gronwall lemma yields  
\begin{equation*}
  \|e^\eps(t)\|_{L^1(\R^n)} \le \|e^\eps(0)\|_{L^1(\R^n)}\exp\(Ct\).
\end{equation*}
Finally, $(e^\eps(0))_{\eps}$ is bounded in $L^1(\R^n)$. This is
obvious for the first two terms of $e^\eps$. For $q^\eps$, a rough
estimate yields:
\begin{align*}
\lA q^\eps(0)\rA_{L^2}
&\le\lA \eps^{-1}\(\la a^\eps(0)\ra^2 -\la a(0)\ra^2\)\rA_{L^2}
\lA \mathcal{Q}_{2}\( a^\eps(0), a(0)\)\rA_{L^\infty},
\end{align*}
and the assumption $a^\eps_{0}-a_{0}=\mathcal{O}(\eps)$ in $H^\infty$
shows that 
\begin{equation*}
  \sup_{0<\eps\le 1}\|e^\eps(0)\|_{L^1(\R^n)}<\infty.
\end{equation*}
This completes the proof of Theorem~\ref{theo:tout}.

\subsection{Convergence of position and current densities}
\label{sec:quad}
As we have already mentioned, Theorem~\ref{theo:main} is a consequence
of the first part of \eqref{eq:uniform2}. Proposition~\ref{prop:quad}
follows from both informations in \eqref{eq:uniform2}. Indeed,
\eqref{eq:uniform2} implies the ``usual'' 
modulated energy estimate, as in 
\cite{BrenierCPDE,ZhangSIMA,LinZhang} (see also
\cite{AC-perte}). The boundedness of $q^\eps$ in $C([0,T];L^2)$, and
the convexity argument \eqref{eq:Qpardessous}, yield
\begin{equation*}
  \sup_{t\in[0,T]} \int_{\R^n} \( |a^\eps(t,x)|^2 - |a(t,x)|^2\)^2 \(
  |a^\eps(t,x)|^{2\si-2}+ |a(t,x)|^{2\si-2}\)^2dx \les \eps^2. 
\end{equation*}
Therefore,
\begin{equation}\label{eq:14h26}
  \sup_{t\in[0,T]} \int_{\R^n} \( |a^\eps(t,x)|^2 -
  |a(t,x)|^2\)^{\si+1} dx \les \eps^2.  
\end{equation}
This yields the first part of Proposition~\ref{prop:quad}, along with
a bound on the rate of convergence as $\eps \to 0$. For the current
density, write
\begin{equation*}
  \IM\(\eps \overline u^\eps \nabla u^\eps\) = |a^\eps|^2\nabla \phi +
  \IM\(\eps \overline a^\eps \nabla a^\eps\). 
\end{equation*}
Since $\nabla \phi\in L^\infty([0,T]\times\R^n)$, \eqref{eq:14h26} yields
\begin{equation*}
 |a^\eps|^2\nabla \phi \Tend \eps 0 |a|^2\nabla \phi\quad \text{in
  }C([0,T];L^{\si+1}).  
\end{equation*}
On the other hand, since $a^\eps$ is bounded in $C([0,T];H^1)$, we
have:
\begin{equation*}
\IM\(\eps \overline a^\eps \nabla a^\eps\)  =\O(\eps)   \quad \text{in
  }C([0,T];L^{1}). 
\end{equation*}
This completes the proof of Proposition~\ref{prop:quad}. 
\begin{remark}
Since we have used \eqref{eq:uniform2} with $k=1$ only, we could also
refine the statements of Proposition~\ref{prop:quad} when $k\ge 2$ is
allowed in  \eqref{eq:uniform2}. 
\end{remark}

\section{Proof of Theorem~\ref{theo:ghost}}
\label{sec:ghost}

To prove Theorem~\ref{theo:ghost}, resume the
approach of E.~Grenier \cite{Grenier98}. His idea was to seek 
\begin{equation*}
  u^\eps(t,x) =a^\eps(t,x)e^{i\phi^\eps(t,x)/\eps},
\end{equation*}
where the pair $U^\eps =(a^\eps,\nabla\phi^\eps)$ is given by a system
of the form \eqref{system:hyp} (with $E\equiv 0$). The point is that
the form \eqref{system:hyp} for this $U^\eps$ meets all the
requirements that we have 
listed, if and only if the nonlinearity is defocusing, and cubic at the
origin. In the case of the homogeneous nonlinearity of
\eqref{eq:nlssemi}, the only admissible case is then $\si=1$. The
second step of the analysis in \cite{Grenier98} consists in showing
that under suitable assumptions, $a^\eps$ and $\phi^\eps$ have an
asymptotic expansion of the form
\begin{equation*}
  a^\eps \Eq \eps 0 a + \eps a^{(1)} + \eps^2 a^{(2)}+\ldots\quad
  ;\quad \phi^\eps \Eq \eps 0 \phi + \eps \phi^{(1)} + \eps^2
  \phi^{(2)}+\ldots 
\end{equation*}
The pair $(a,\phi)$ solves (the analogue of) \eqref{eq:limite}. Note
that because the phase $\phi^\eps$ is divided by $\eps$, we need to
take $\phi^{(1)}$ into account in order to have a point-wise
description of $u^\eps$:
\begin{equation*}
  u^\eps \Eq \eps 0 a e^{i\phi^{(1)}} e^{i\phi/\eps}.  
\end{equation*}
Therefore, the rapidly oscillatory phase for $u^\eps$ is given by
$\phi$, and its amplitude at leading order is given by $a
e^{i\phi^{(1)}}$ (which does not depend on $\eps$). If $u^\eps$ solves
\begin{equation*}
  i\eps \d_{t}u^\eps  + \frac{\eps^2}{2}\Delta u^\eps = f\(\la
u^\eps\ra^{2}\)u^\eps \quad ;\quad u^\eps_{\arrowvert
  t=0}=a_0^\eps e^{i\phi_{0}/\eps}, 
\end{equation*}
where $a_0^\eps$ satisfies Assumption~\ref{assu:ghost}, then
$\phi^{(1)}$ is given by the system
\begin{equation*}
\left\{
  \begin{aligned}
    \d_t \phi^{(1)} +\nabla \phi \cdot \nabla \phi^{(1)} +
    2\RE\left(\overline a a^{(1)}\right)f'\left(
    |a|^2\right)&=0,\\ 
   \d_t a^{(1)} +\nabla\phi\cdot \nabla a^{(1)} + \nabla
   \phi^{(1)}\cdot \nabla a + \frac{1}{2} a^{(1)}\Delta \phi
   +\frac{1}{2}a\Delta \phi^{(1)}      &= \frac{i}{2}\Delta a,\\
\phi^{(1)}\big|_{t=0}=0\quad ; \quad a^{(1)}\big|_{t=0}=a_1.
  \end{aligned}
\right.
\end{equation*}
This coupling shows that $\phi^{(1)}$ is a (nonlinear) function of
$a,\phi$, and $a_1$, the term of order $\eps$ in the expansion of the
initial data $a_0^\eps$. In our case, $f(y)=y^\si$: we introduce the system
\begin{equation}\label{eq:ghostsyst}
\left\{
  \begin{aligned}
    \d_t \phi^{(1)} +\nabla \phi \cdot \nabla \phi^{(1)} +
    2\si \RE\left(\overline a a^{(1)}\right)
    |a|^{2\si-2}&=0,\\ 
   \d_t a^{(1)} +\nabla\phi\cdot \nabla a^{(1)} + \nabla
   \phi^{(1)}\cdot \nabla a + \frac{1}{2} a^{(1)}\Delta \phi
   +\frac{1}{2}a\Delta \phi^{(1)}      &= \frac{i}{2}\Delta a,\\
\phi^{(1)}\big|_{t=0}=0\quad ; \quad a^{(1)}\big|_{t=0}=a_1.
  \end{aligned}
\right.
\end{equation}

\begin{lemma}\label{weaklemma:MUK2}
Let $n\ge 1$, and let Assumption~\ref{assu:ghost} be
satisfied. Then \eqref{eq:ghostsyst} has a
unique solution 
$(\phi^{(1)},a^{(1)})$ in $C([0,T^*[;H^\infty(\R^n))$, where $T^*$
is given by Lemma~$\ref{weaklemma:MUK}$.   
\end{lemma}
\begin{proof}
Again, at the zeroes of $a$, \eqref{eq:ghostsyst} ceases to be
hyperbolic, and 
we cannot solve the Cauchy problem by a standard argument. The
strategy of the proof is  
to transform the equations so as to obtain an auxiliary hyperbolic
system for $(\nabla\phi^{(1)},A_1)$  
for some good unknown $A_1$, depending linearly upon $a^{(1)}$. The
definition of $A_1$ depends on the parity of $\si$.  
This allows to determine a function $\phi^{(1)}$ and next to define a
function $a^{(1)}$ by solving the  
second equation in \eqref{eq:ghostsyst}. We conclude the proof by 
checking that $(\phi^{(1)},a^{(1)})$ does solve
\eqref{eq:ghostsyst}. The first change of unknown consists in
considering $v_1 \defn  
\nabla \phi^{(1)}$. The first equation in \eqref{eq:ghostsyst} yields:
\begin{equation*}
  \d_t v_1 + v \cdot \nabla v_1 + 2\si \nabla \RE
  \(|a|^{2\si-2}  \overline a a^{(1)}\) = - v_1\cdot \nabla v, 
\end{equation*}
where we have denoted $v=\nabla \phi$. 

\smallbreak
\noindent {\bf First case: $\si\ge 2$ is even}. 
Consider the new unknown
\begin{equation*}
 A_1\defn |a|^{\si-2} \RE\(\overline a
a^{(1)}\).
\end{equation*}
We check that, if $(\phi^{(1)},a^{(1)})$ solves \eqref{eq:ghostsyst}, then 
\begin{equation}
  \label{eq:10h01}
  \left\{
    \begin{aligned}
     \d_t v_1 + v \cdot \nabla v_1 + 2\si |a|^{\si}\nabla  A_1 &= - v_1\cdot
\nabla v- 2\si  A_1 \nabla\( |a|^{\si}\),\\
 \d_t A_1 +v \cdot \nabla A_1 +\frac{1}{2}|a|^{\si}\DIV v_1 & =
-\frac{1}{\si}\nabla \(|a|^{\si}\)\cdot v_1 -\frac{\si}{2}A_1\DIV v\\
&\ +\frac{i}{2} \RE\( |a|^{\si-2}\overline a \Delta a\).  
    \end{aligned}
\right.
\end{equation}
This linear system is hyperbolic
symmetric, and its coefficients are 
smooth since $\si\in 2\N$ and $a,v\in
C^\infty([0,T^*[;H^\infty(\R^n))$,  from 
Lemma~\ref{lemma:MUK}. 
In particular, uniqueness for \eqref{eq:ghostsyst} follows from the
uniqueness for \eqref{eq:10h01}. Note that, since $\si-2\in 2\N$, 
\begin{equation*}
  (v_1,A_1)\big|_{t=0} = \(0,|a_0|^{\si-2}\RE\(\overline a_0 a_1\)\)\in
  H^\infty \(\R^n\)^2.
\end{equation*}
Therefore, \eqref{eq:10h01} possesses a unique solution in
$C^\infty([0,T^*[;H^\infty(\R^n))$. We next define $\phi^{(1)}\in
C^{\infty}([0,T^*[:H^\infty(\R^n))$ by 
\begin{equation*}
  \phi^{(1)}(t,x) = - \int_0^t \(v(\tau,x)\cdot v_1(\tau,x) + 
2\sigma |a(\tau,x)|^{\si}
  A_1(\tau,x)\)\, d\tau. 
\end{equation*}
Then $\d_t \(\nabla \phi^{(1)} -v_1\)= 0$, therefore
$v_1=\nabla \phi^{(1)}$ and hence $\phi^{(1)}$ satisfies
 \begin{equation*}
    \d_t \phi^{(1)} +v\cdot \nabla \phi^{(1)} + 
2\si |a|^{\si} A_1=0,\quad 
\phi^{(1)}\big|_{t=0}=0.
\end{equation*}
Once this is
granted, we can define $a^{(1)}\in
C^{\infty}([0,T^*[:H^\infty(\R^n))$ as the unique solution
of the linear equation
\begin{equation*}
\left\{
  \begin{aligned}
  & \d_t a^{(1)} +v\cdot \nabla a^{(1)} + \nabla
   \phi^{(1)}\cdot \nabla a + \frac{1}{2} a^{(1)}\DIV v 
   +\frac{1}{2}a\Delta \phi^{(1)} = \frac{i}{2}\Delta a,\\
&a^{(1)}\big|_{t=0}=a_1.
  \end{aligned}
\right.
\end{equation*}
By construction, $A_1$ and $|a|^{\si-2} \RE\( \overline a
a^{(1)}\)$ solve the same linear equation, where $\phi^{(1)}$ is
viewed as a smooth coefficient. Therefore, these two functions
coincide, and $(\phi^{(1)},a^{(1)})$  solves 
\eqref{eq:ghostsyst}. 

\smallbreak

\noindent{\bf Second case: $\si$ is odd}. In this case, $\si=2m+1$,
for some $m\in 
\N$. We consider the new unknown 
\begin{equation*}
  A_1\defn |a|^{\si-1}a^{(1)}=|a|^{2m}a^{(1)}. 
\end{equation*}

We check that $ (v_1,A_1)$ must solve 
\begin{equation*}
  \left\{
    \begin{aligned}
     \d_t v_1 + v \cdot \nabla v_1 + 2\si \RE\(|a|^{2m}\overline a
     \nabla A_1\)& = - v_1\cdot 
\nabla v- 2\si \RE\( A_1 \nabla\( |a|^{2m}\overline a\)\),\\
 \d_t A_1 +v \cdot \nabla A_1 +\frac{1}{2}|a|^{2m}a\DIV v_1 & =
-\frac{\si}{2} A_1\DIV v- |a|^{2m}\nabla a\cdot v_1\\
&\  +\frac{i}{2} |a|^{2m} \Delta a.  
    \end{aligned}
\right.
\end{equation*}
We can then conclude as in the first case, by considering
$(v_1,A_1,\overline A_1)$.
\end{proof}

Theorem~\ref{theo:ghost} follows from:
\begin{proposition}
  Let $n\le 3$, and let Assumption~\ref{assu:ghost} be satisfied. Set
  $  \widetilde  
  a \defn a e^{i\phi^{(1)}} $. Then for any  $T\in ]0,T^*[$, there
  exists $\eps(T)>0$ such that $a^\eps \in C([0,T];H^\infty)$ for
  $\eps\in ]0,\eps(T)]$, and
  \begin{equation*}
    \lA \tm -\widetilde a \rA_{L^\infty([0,T];H^k)} =\O(\eps), 
  \end{equation*}
where $k$ is as in Theorem~\ref{theo:tout}. 
\end{proposition}
\begin{proof}
 Since the proof follows the same lines as the proof of
 Theorem~\ref{theo:main}, we shall indicate its main steps only. Denote
 \begin{equation*}
   \reste = a^\eps -\widetilde a\quad ;\quad  \widetilde a^{(1)} =
   a^{(1)} e^{i\phi^{(1)}}. 
 \end{equation*}
From \eqref{eq:limite}, \eqref{eq:ae} and \eqref{eq:ghostsyst}, we see
that $\reste$ solves 
\begin{equation*}
  \left\{
    \begin{aligned}
      \d_t \reste + \tv\cdot \nabla \reste + \frac{1}{2}\reste\DIV \tv
      - i\frac{\eps}{2}\Delta \reste &= i\frac{\eps}{2}\Delta
      \widetilde a -i S^\eps,\\
\reste_{\mid t=0} = a_0^\eps -a_0 = \eps a_1 +\O\(\eps^2\),
    \end{aligned}
\right.
\end{equation*}
where the term $S^\eps$ is given by:
\begin{equation*}
  S^\eps = \frac{1}{\eps}\(\la a^\eps \ra^{2\si} -\la \widetilde a
  \ra^{2\si} \) 
  a^\eps -2\si \widetilde a|\widetilde a|^{2\si -2}\RE\( \overline
  {\widetilde a}
  \widetilde a^{(1)}\).  
\end{equation*}
We check that for all $s\ge 0$, we have, in $H^s(\R^n)$:
\begin{equation*}
  S^\eps = \frac{1}{\eps}\(\la a^\eps \ra^{2\si} -\la \widetilde a
  +\eps \widetilde a^{(1)} \ra^{2\si} \) 
  a^\eps +2\si \reste |\widetilde a|^{2\si -2}\RE\( \overline
  {\widetilde a}
  \widetilde a^{(1)}\) + \O(\eps).  
\end{equation*}
The last term should be viewed as a small source term. The second one
is linear in $\reste$, and is suitable in view of an application of
the Gronwall Lemma. There remains to handle the first term.  
At this stage, we can mimic the approach detailed in
\S\ref{sec:reste}. Introduce the nonlinear change of unknown:
\begin{equation*}
  \widetilde q^\eps = \frac{1}{\eps} B_\si \( \la a^\eps\ra^2,\la
  \widetilde a 
  +\eps \widetilde a^{(1)}\ra^2\) \quad ;\quad \widetilde g^\eps =
  G_\si \( \la a^\eps\ra^2,\la  \widetilde a
  +\eps \widetilde a^{(1)}\ra^2\),
\end{equation*}
where $B_\si$ and $G_\si$ are defined in Notation~\ref{nota:GBP}. We
check that $(\reste,\nabla\reste,\widetilde q^\eps)$ solves a system
of the form \eqref{eq:systfinal}, plus some extra source terms of
order $\O(\eps)$ in $H^s(\R^n)$. We also note that the initial data
are of order $\O(\eps)$, from Assumption~\ref{assu:ghost}:
\begin{equation*}
  \lA(\reste,\nabla\reste)\bigr|_{ t=0}\rA_{H^s} =
  \O(\eps),\ \forall s\ge 0. 
\end{equation*}
We also have 
\begin{equation*}
  \lA\widetilde q^\eps\bigr|_{ t=0}\rA_{H^{k-1}}=\O(\eps),
\end{equation*}
where $k$ is as Theorem~\ref{theo:tout}. 

Following the approach of \S\ref{sec:reste}, the proposition stems
from Gronwall lemma. Note also that the time $T$ can be taken
arbitrarily close to $T^*$, by the usual continuity argument, since we
now have an error estimate that goes to zero with $\eps$. 
\end{proof}
To conclude this paragraph, we note that unless $a_0$ is real valued and
  $a_1\in i\R$, one must not expect $\widetilde a=a$. Indeed, we see
  that 
  \begin{equation*}
    \phi^{(1)}_{\mid t=0}=0\quad ;\quad \d_t \phi^{(1)}_{\mid
    t=0}=-2\si \RE\(\overline a_0a_1\)|a_0|^{2\si-2}.
  \end{equation*}
So in general, $\phi^{(1)}\not \equiv 0$, and $\widetilde a\not =a$. 
On the other hand if $a_0$ is real-valued, then so is $a$. In this
case,
\begin{equation*}
  \IM\(\overline a \Delta a \)\equiv 0,
\end{equation*}
and $(\phi^{(1)},\RE (\overline a 
a^{(1)}))$ solves an 
homogeneous linear system. Therefore, if $\RE (\overline a
a^{(1)})=0$ at time $t=0$, then $\phi^{(1)}\equiv 0$.

\section{Further remarks}
\label{sec:remarks}

The following remarks serve to clarify some features of 
the systems we produced. 

\subsection{Regularity of the initial data}

It is a matter of routine to extend the previous analysis to the case
where the initial data belong  
to $H^s(\R^n)$ with $s<+\infty$ large enough.

\subsection{Introducing an external potential}\label{sec:potential} 

To treat a possibly more physically relevant case, one might want to
consider \eqref{eq:nlssemi} with an extra external potential:
\begin{equation*}
i\eps \partial_{t} \tu + \frac{\eps^2}{2}\Delta \tu = V \tu + \la
\tu\ra^{2\sigma}\tu\quad ;\quad \tu_{\arrowvert
  t=0}=a_0^\eps e^{i\phi_{0}/\eps}, 
\end{equation*}
where $V=V(t,x)$ is real-valued, and possibly time-dependent. 
As noticed in \cite{CaBKW}, it is sensible to consider an external
potential $V$ and an initial phase $\phi_0$ which are smooth and
sub-quadratic:
\begin{equation*}
  \d_x^\alpha V\in C(\R;L^\infty(\R^n)),\ \d^\alpha \phi_0\in
  L^\infty(\R^n),\ \forall \alpha \in \N^n, \ |\alpha |\ge 2. 
\end{equation*}
This includes the case of the harmonic oscillator, commonly used in
the theory of Bose--Einstein condensation (\cite{KNSQ}).  
The main remark in \cite{CaBKW} is that the introduction of this
assumption does not deeply change the analysis. Indeed, we can resume
the analysis of \eqref{eq:nlssemi}: introduce the solution to the
standard eikonal equation
\begin{equation*}
  \partial_t \phi_{\rm eik} +\frac{1}{2}|\nabla_x \phi_{\rm eik}|^2
    +V=0\quad ;\quad 
    \phi_{{\rm eik}}\bigr|_{ t=0}=\phi_0\, .
\end{equation*}
Decomposing the phase $\phi$ of the above quasi-linear analysis as
\begin{equation*}
  \phi = \phi_{\rm eik} +\underline \phi,
\end{equation*}
and seeking $\underline \phi$ in Sobolev spaces, we see that the extra
terms appearing after this sort of linearization can be treated like
semi-linear terms. Therefore, mimicking the above computations, and
using only extra perturbative arguments, it is easy to adapt
Theorems~\ref{theo:main} and \ref{theo:ghost} to this case. 

\subsection{About conservation laws}
\label{sec:conserv}

Recall some important evolution laws for \eqref{eq:nlssemi}:
\begin{align*}
  \text{Mass: }& \frac{d}{dt}\|u^\varepsilon(t)\|_{L^2}=0\, .\\
\text{Energy: }& \frac{d}{dt}\left(\frac{1}{2}\|\varepsilon\nabla_x
  u^\varepsilon\|_{L^2}^2 
+ \frac{1}{\si+1}\|u^\varepsilon\|_{L^{2\si+2}}^{2\si+2} \right)=0\, .\\
\text{Momentum: }&\frac{d}{dt}\IM\int \overline
u^\varepsilon(t,x) 
\varepsilon \nabla_x u^\varepsilon(t,x)dx =0\, .\\ 
\text{Pseudo-conformal law: }& \frac{d}{dt}\left(\frac{1}{2}\|J^\varepsilon(t)
u^\varepsilon\|_{L^2}^2 
+ \frac{t^2}{\si+1}\|u^\varepsilon\|_{L^{2\si+2}}^{2\si+2}
\right)\\
&= \frac{t}{\si+1}(2-n\si)\|u^\varepsilon\|_{L^{2\si+2}}^{2\si+2} ,
\end{align*}
where $J^\varepsilon(t) = x + i\varepsilon t\nabla_x$. These
evolutions are deduced 
from the usual ones ($\varepsilon =1$, see
e.g. \cite{CazCourant,Sulem}) via the scaling 
$\psi(t,x)= u(\varepsilon t,\varepsilon x)$. Writing $u^\eps = a^\eps
e^{i\phi/\eps}$, and passing to 
the limit formally in the above formulae yields:
\begin{align*}
   \frac{d}{dt}\|a(t)\|_{L^2}&=0\, .\\ 
 \frac{d}{dt}\int \left( \frac{1}{2}|a(t,x)|^2 |\nabla \phi(t,x)|^2 +
\frac{1}{\si+1}|a(t,x)|^{2\si+2}\right)dx &=0\, .\\ 
\frac{d}{dt}\int |a(t,x)|^2 \nabla \phi(t,x)dx &=0\, .\\ 
 \frac{d}{dt}\int \left(\frac{1}{2} \left| \left(x -t \nabla
  \phi(t,x)\right)a(t,x)\right|^2 +\frac{t^2}{\si+1} |a(t,x)|^{2\si+2} 
\right)&dx=\\
=\frac{t}{\si+1}(2-n\si) &\int |a(t,x)|^{2\si+2}dx\, .
\end{align*}
Note that we also
have the conservation (\cite{CN}):
\begin{equation*}
   \frac{d}{dt}\RE\int \overline u^\varepsilon(t,x)
 J^\varepsilon(t)u^\varepsilon(t,x) dx=0 \, ,
\end{equation*}
which yields:
\begin{equation*}
  \frac{d}{dt}\int \left(x -t \nabla
  \phi(t,x)\right)|a(t,x)|^2 dx =0\, . 
\end{equation*}
All these expressions involve only $(|a|^2, \nabla \phi)=(|\widetilde
a|^2, \nabla \phi)$. Recall that 
if we set $(\rho,v)=(|a|^2, \nabla \phi)$, then \eqref{eq:limite}
implies 
\begin{equation}\label{eq:eulerMUK}
  \left\{
    \begin{aligned}
     &\d_t v +v\cdot \nabla v + \nabla \(\rho^\si\)=0\quad &&;\quad
v_{\mid t =0}=\nabla\phi_{0},\\ 
& \d_t \rho + \DIV \(\rho v\) =0\quad  && ;\quad \rho_{\mid
  t=0}=|a_0|^2. 
    \end{aligned}
\right.
\end{equation}
Rewriting the above evolution laws, we get:
\begin{align}
   \frac{d}{dt}\int_{R^n}\rho(t,x)dx &=0\, .\notag\\ 
 \frac{d}{dt}\int \left( \frac{1}{2}\rho(t,x) |v(t,x)|^2 +
\frac{1}{\si+1}\rho(t,x)^{\si+1}\right)dx &=0  :\text{ energy}.\notag\\  
\frac{d}{dt}\int \rho(t,x) v(t,x)dx &=0\,.\notag\\ 
 \frac{d}{dt}\int \left(\frac{1}{2} \left| \left(x -t 
  v(t,x)\right)\right|^2 \rho(t,x) +\frac{t^2}{\si+1} \rho(t,x)^{\si+1} 
\right)&dx=\notag\\
=\frac{t}{\si+1}(2-n\si) &\int \rho(t,x)^{\si+1}dx\, .\label{eq:pseudoconf} \\
\frac{d}{dt}\int \left(x -t v (t,x)\right)\rho(t,x) dx &=0\,
 .\notag
\end{align}
We thus retrieve formally some evolution laws
for the compressible Euler equation \eqref{eq:eulerMUK} (see
e.g. \cite{Serre97,Xin98}), with the pressure law
$p(\rho)=c \rho^{\si+1}$.

\subsection{About global in time results}
\label{sec:blowup}
We point out that the solution to \eqref{eq:limite} must not be
expected to be smooth for all time: the time $T^*$ in
Lemma~\ref{lemma:MUK} is finite in general. Recall that 
$(\rho,v)=(|a|^2,\nabla\phi)$ solves
\eqref{eq:eulerMUK}. 
Theorem~$3$ in~\cite{MUK86} (see also \cite{Xin98}) implies that, if
$\nabla \phi_0$ and $|a_0|^2$ are  
compactly supported, then the life span $T^*$ in Lemma~\ref{lemma:MUK}
is necessarily finite. Note that these initial data can be chosen 
arbitrarily small: the phenomenon remains. 
\begin{proposition}\label{prop:blowup}
Let $n\ge1$ and $\sigma\ge 1$. For all initial data 
$(a_{0},\phi_{0})\in C^2(\R^n)$ with compact support, there does not
exist $(a,\phi)\in C^2([0,+\infty[\times \R^n)$  
satisfying the Cauchy problem~\eqref{eq:limite}.
\end{proposition}
A word of caution: because of one technical assumption in the
definition of regular solution in~\cite{MUK86},  
Theorem~3 in \cite{MUK86} does not apply directly. 
Yet, one can prove our claim by combining the proof 
of Lemma~\ref{lemma:MUK} with the approach in~\cite{MUK86}. 
Indeed, recall that $U\defn (a^\sigma,\nabla\phi)$ satisfies 
$\partial_{t}U + \sum A_{j}(U)\partial_{j}U=0$ where the $A_{j}$'s are
$n\times n$ matrices linear in their argument. Therefore,   
the proof of Theorem~$2$ in \cite{MUK86} shows that $U$ is compactly
supported, and so is $(\rho,v)\defn (|a|^2,\nabla \phi)$,  
with support included in the support of $(\la a_{0}\ra^2,\nabla
\phi_{0})$. And this is the only point which requires the above
mentioned  technical assumption.
\smallbreak

Note also that the proof of this result in \cite{Xin98} relies on the
evolution law for the total pressure
\begin{equation}\label{eq:pression}
  \int_{\R^n} p(t,x)dx=\int_{\R^n} \rho(t,x)^{\si+1}dx .  
\end{equation}
This approach is very similar to the Zakharov--Glassey method
\cite{Z,Glassey}, which yields a sufficient condition for the finite time
blow-up of solutions to the \emph{focusing} nonlinear Schr\"odinger
equation. As noticed by M.~Weinstein \cite{Weinstein83}, the identity
used by Zakharov, 
and generalized by Glassey, follows from the pseudo-conformal law,
along with the conservation of energy. For $\si\ge 2/n$ and a
defocusing nonlinearity, this approach yields an upper bound for the
$L^2$-norm of $x u$, the momentum of $u$. When this upper bound may become
negative, finite time blow-up occurs. 

In the present context, the nonlinearity is defocusing, but the idea
is similar. Note that (the generalized version of) \eqref{eq:pseudoconf}
is the key ingredient in 
the proof of Z.~Xin \cite{Xin98} (Z.~Xin considers
Navier--Stokes equations). Expanding \eqref{eq:pseudoconf}, and using
the conservation of energy, we recover an upper bound for
\eqref{eq:pression} which goes to
zero as $t\to \infty$. But so long
as $v$ remains bounded, 
\eqref{eq:eulerMUK} is an ordinary differential equations for $\rho$,
thus contradicting the upper bound for \eqref{eq:pression}, 
unless $v$ ceases to be smooth in finite time (see \cite{Xin98} for
the details).

\subsection{About focusing nonlinearities}

The main feature of the limit system we used is that it enters, up to
a change of unknowns,  
into the framework of quasi-linear hyperbolic systems. This comes from 
the fact that we consider the defocusing case. 
Had we worked instead with the focusing case, where $+|u|^{2\sigma}u$
is replaced with  
$-|u|^{2\sigma}u$, the corresponding limit system would have been
ill-posed. We refer to \cite{GuyCauchy}, in which G. M\'etivier 
establishes Hadamard's instabilities for non-hyperbolic nonlinear equations.
\smallbreak

As an example, consider the Cauchy problem
\begin{equation}\label{syst:elliptic}
  \left\{
    \begin{aligned}
&\d_t \phi +\frac{1}{2}|\partial_{x} \phi|^2- |a|^{2\si}=0
\quad &&;\quad \phi_{\mid t =0}=\phi_{0},\\
&\d_t a +\partial_{x}\phi\partial_{x} a +\frac{1}{2}a
\partial_{x}^2 \phi = 0\quad   &&;\quad a_{\mid
  t=0}=a_0.
 \end{aligned}
\right.
\end{equation}
The following result follows from
Hadamard's argument (see \cite{GuyCauchy}).

\begin{proposition}
Suppose that $(\phi,a)$ in $C^2([0,T]\times \R)$ solves
\eqref{syst:elliptic}. 
If $\phi_{0}(x)$ is real analytic near $\underline{x}$ and if
$a_{0}(\underline x)>0$, 
then $a_{0}(x)$ is real analytic near $\underline{x}$. 
Consequently, there are smooth initial data for which the Cauchy
problem has no solution. 
\end{proposition}

This shows that to study the semi-classical limit for
the focusing analogue of \eqref{eq:nlssemi}, working with analytic
data, as in \cite{PGX93,ThomannAnalytic}, is not only convenient:
it is necessary. 

\providecommand{\bysame}{\leavevmode\hbox to3em{\hrulefill}\thinspace}
\providecommand{\MR}{\relax\ifhmode\unskip\space\fi MR }
\providecommand{\MRhref}[2]{%
  \href{http://www.ams.org/mathscinet-getitem?mr=#1}{#2}
}
\providecommand{\href}[2]{#2}


\begin{thebibliography}{10}

\bibitem{ThomasARMA}
T.~Alazard, \emph{Low {M}ach number limit of the full {N}avier-{S}tokes
  equations}, Arch. Ration. Mech. Anal. \textbf{180} (2006), no.~1, 1--73.

\bibitem{AC-perte}
T.~Alazard and R.~Carles, \emph{Loss of regularity for
  super-critical nonlinear {S}chr\"odinger equations}, preprint: {\tt
  math.AP/0701857}.

\bibitem{Alinhac}
S.~Alinhac, \emph{Blowup for nonlinear hyperbolic equations}, Birkh\"auser
  Boston Inc., Boston, MA, 1995.

\bibitem{AlinhacAJM95}
\bysame, \emph{Explosion g\'eom\'etrique pour des syst\`emes
  quasi-lin\'eaires}, Amer. J. Math. \textbf{117} (1995), no.~4, 987--1017.

\bibitem{AlinhacForges}
\bysame, \emph{A minicourse on global existence and blowup of classical
  solutions to multidimensional quasilinear wave equations}, Journ\'ees
  ``\'Equations aux D\'eriv\'ees Partielles'' (Forges-les-Eaux, 2002), Univ.
  Nantes, Nantes, 2002, pp.~Exp. No. I, 33.

\bibitem{BrenierCPDE}
Y.~Brenier, \emph{Convergence of the {V}lasov-{P}oisson system to the
  incompressible {E}uler equations}, Comm. Partial Differential Equations
  \textbf{25} (2000), no.~3-4, 737--754.

\bibitem{BurqMesures}
N.~Burq, \emph{Mesures semi-classiques et mesures de d\'efaut}, Ast\'erisque
  (1997), no.~245, Exp.\ No.\ 826, 4, 167--195, S\'eminaire Bourbaki, Vol.\
  1996/97.

\bibitem{BGTENS}
N.~Burq, P.~G\'erard, and N.~Tzvetkov, \emph{Multilinear eigenfunction
  estimates and global existence for the three dimensional nonlinear
  {S}chr\"odinger equations}, Ann. Sci. \'Ecole Norm. Sup. (4) \textbf{38}
  (2005), no.~2, 255--301.

\bibitem{CaARMA}
R.~Carles, \emph{Geometric optics and instability for semi-classical
  {S}chr\"odinger equations}, Arch. Ration. Mech. Anal. \textbf{183} (2007),
  no.~3, 525--553.

\bibitem{CaInstab}
\bysame, \emph{On instability for the cubic nonlinear {S}chr\"odinger
  equation}, C. R. Math. Acad. Sci. Paris (2007), archived at {\tt
  math.AP/0701858}.

\bibitem{CaBKW}
\bysame, \emph{{WKB} analysis for nonlinear {S}chr\"odinger equations with
  potential}, Comm. Math. Phys. \textbf{269} (2007), no.~1, 195--221.

\bibitem{CN}
R.~Carles and Y.~Nakamura, \emph{Nonlinear {S}chr\"odinger equations with
  {S}tark potential}, Hokkaido Math. J. \textbf{33} (2004), no.~3, 719--729.

\bibitem{CazCourant}
T.~Cazenave, \emph{Semilinear {S}chr\"odinger equations}, Courant Lecture Notes
  in Mathematics, vol.~10, New York University Courant Institute of
  Mathematical Sciences, New York, 2003.

\bibitem{CazHar}
T.~Cazenave and A.~Haraux, \emph{An introduction to semilinear evolution
  equations}, Oxford Lecture Series in Mathematics and its Applications,
  vol.~13, The Clarendon Press Oxford University Press, New York, 1998,
  Translated from the 1990 French original by Yvan Martel and revised by the
  authors.

\bibitem{JYC90}
J.-Y. Chemin, \emph{Dynamique des gaz \`a masse totale finie}, Asymptotic Anal.
  \textbf{3} (1990), no.~3, 215--220.

\bibitem{CheverryCMP}
C.~Cheverry, \emph{Propagation of oscillations in real vanishing viscosity
  limit}, Comm. Math. Phys. \textbf{247} (2004), no.~3, 655--695.

\bibitem{CheverryBullSMF}
\bysame, \emph{Cascade of phases in turbulent flows}, Bull. Soc. Math. France
  \textbf{134} (2006), no.~1, 33--82.

\bibitem{CG05}
C.~Cheverry and O.~Gu\`es, \emph{Counter-examples to concentration-cancellation
  and supercritical nonlinear geometric optics for the incompressible {E}uler
  equations}, Arch. Ration. Mech. Anal. (2007), to appear.

\bibitem{CKSTTAnnals}
J.~Colliander, M.~Keel, G.~Staffilani, H.~Takaoka, and T.~Tao, \emph{Global
  well-posedness and scattering for the energy--critical nonlinear
  {S}chr\"odinger equation in {$\mathbb R\sp 3$}}, Ann. of Math. (2), to
  appear.

\bibitem{DesjardinsLin}
B.~Desjardins and C.-K. Lin, \emph{On the semiclassical limit of the general
  modified {NLS} equation}, J. Math. Anal. Appl. \textbf{260} (2001), no.~2,
  546--571.

\bibitem{GasserLinMarkowich}
I.~Gasser, C.-K. Lin, and P.~A. Markowich, \emph{A review of dispersive limits
  of (non)linear {S}chr\"odinger-type equations}, Taiwanese J. Math. \textbf{4}
  (2000), no.~4, 501--529.

\bibitem{PGX93}
P.~G{\'e}rard, \emph{Remarques sur l'analyse semi-classique de l'\'equation de
  {S}chr\"odinger non lin\'eaire}, S\'eminaire sur les \'Equations aux
  D\'eriv\'ees Partielles, 1992--1993, \'Ecole Polytech., Palaiseau, 1993,
  pp.~Exp.\ No.\ XIII, 13.

\bibitem{GMMP}
P.~G{\'e}rard, P.~A. Markowich, N.~J. Mauser, and F.~Poupaud,
  \emph{Homogenization limits and{W}igner transforms}, Comm. Pure Appl. Math.
  \textbf{50} (1997), no.~4, 323--379.

\bibitem{GV85c}
J.~Ginibre and G.~Velo, \emph{The global {C}auchy problem for the nonlinear
  {S}chr\"odinger equation revisited}, Ann. Inst. H. Poincar\'e Anal. Non
  Lin\'eaire \textbf{2} (1985), 309--327.

\bibitem{Glassey}
R.~T. Glassey, \emph{On the blowing up of solutions to the {C}auchy problem for
  nonlinear {S}chr\"odinger equations}, J. Math. Phys. \textbf{18} (1977),
  1794--1797.

\bibitem{Grenier98}
E.~Grenier, \emph{Semiclassical limit of the nonlinear {S}chr\"odinger equation
  in small time}, Proc. Amer. Math. Soc. \textbf{126} (1998), no.~2, 523--530.

\bibitem{KNSQ}
E.~B. Kolomeisky, T.~J. Newman, J.~P. Straley, and X.~Qi, \emph{Low-dimensional
  {B}ose liquids: Beyond the {G}ross-{P}itaevskii approximation}, Phys. Rev.
  Lett. \textbf{85} (2000), no.~6, 1146--1149.

\bibitem{Lebeau05}
G.~Lebeau, \emph{Perte de r\'egularit\'e pour les \'equations d'ondes
  sur-critiques}, Bull. Soc. Math. France \textbf{133} (2005), 145--157.

\bibitem{LinZhang}
F.~Lin and P.~Zhang, \emph{{S}emiclassical limit of the {G}ross-{P}itaevskii
  equation in an exterior domain}, Arch. Rational Mech. Anal. \textbf{179}
  (2005), no.~1, 79--107.

\bibitem{MUK86}
T.~Makino, S.~Ukai, and S.~Kawashima, \emph{Sur la solution \`a support compact
  de l'\'equation d'{E}uler compressible}, Japan J. Appl. Math. \textbf{3}
  (1986), no.~2, 249--257.

\bibitem{GuyCauchy}
G.~M{\'e}tivier, \emph{Remarks on the well-posedness of the nonlinear {C}auchy
  problem}, Geometric analysis of PDE and several complex variables, Contemp.
  Math., vol. 368, Amer. Math. Soc., Providence, RI, 2005, pp.~337--356.

\bibitem{Serre97}
D.~Serre, \emph{Solutions classiques globales des \'equations d'{E}uler pour un
  fluide parfait compressible}, Ann. Inst. Fourier \textbf{47} (1997),
  139--153.

\bibitem{Sulem}
C.~Sulem and P.-L. Sulem, \emph{The nonlinear {S}chr\"odinger equation,
  self-focusing and wave collapse}, Springer-Verlag, New York, 1999.

\bibitem{Taylor3}
M.~Taylor, \emph{Partial differential equations. {III}}, Applied Mathematical
  Sciences, vol. 117, Springer-Verlag, New York, 1997, Nonlinear equations.

\bibitem{ThomannAnalytic}
L.~Thomann, \emph{Instability for {NLS}}, preprint, 2007.

\bibitem{Weinstein83}
M.~I. Weinstein, \emph{Nonlinear {S}chr\"odinger equations and sharp
  interpolation estimates}, Comm. Math. Phys. \textbf{87} (1982/83), no.~4,
  567--576.

\bibitem{Xin98}
Z.~Xin, \emph{Blowup of smooth solutions of the compressible {N}avier-{S}tokes
  equation with compact density}, Comm. Pure Appl. Math. \textbf{51} (1998),
  229--240.

\bibitem{Z}
V.~E. Zakharov, \emph{Collapse of {L}angmuir waves}, Sov. Phys. JETP
  \textbf{35} (1972), 908--914.

\bibitem{ZhangSIMA}
P.~Zhang, \emph{Wigner measure and the semiclassical limit of
  {S}chr\"odinger-{P}oisson equations}, SIAM J. Math. Anal. \textbf{34} (2002),
  no.~3, 700--718.

\end{thebibliography}
\end{document}